\documentclass[12pt]{article}
\pdfoutput=1
\usepackage[utf8]{inputenc}
\usepackage{amssymb,amsthm}
\usepackage{amsmath}
\usepackage[shortlabels]{enumitem}
\usepackage{mathrsfs,color}
\usepackage[left=2cm, right=2cm, top=2cm, bottom=2cm]{geometry}
\usepackage[english]{babel}
\usepackage{csquotes}
\usepackage[backend = biber, style = numeric, sorting = nty, maxbibnames = 10]{biblatex}
\usepackage{standalone}
\usepackage{xr}
\usepackage{graphicx}
\usepackage{float}
\usepackage{caption}
\graphicspath{ {./images/sin_tr/}{./images/exp_tr/}{./images/slow_tr/}{./images/sin_heat/}{./images/exp_heat/} }
\theoremstyle{definition}
\newtheorem{proposition}{Proposition}
\newtheorem{conjecture}{Conjecture}

\newtheorem{theorem}[proposition]{Theorem}

\theoremstyle{definition}

\theoremstyle{definition}
\newtheorem{definition}{Definition}
\theoremstyle{remark}
\newtheorem*{remark}{Remark}
\DeclareMathOperator\erfc{erfc}

\begin{document}

\begin{center}
\textbf{Speed of convergence of Chernoff approximations for two model examples: \\
heat equation and transport equation \\}
Pavel S. Prudnikov \\
psprudnikov@icloud.com
\vskip2mm
National Research University Higher School of Economics\\ 
Nizhny Novgorod City, Russia
\end{center}

\textbf{Abstract.} Paul Chernoff in 1968 proposed his approach to approximations of one-parameter operator semigroups while trying to give a rigorous mathematical meaning to Feynman's path integral formulation of quantum mechanics. In early 2000's Oleg Smolyanov noticed that Chernoff's theorem may be used to obtain approximations to solutions of initial-value problems for linear partial differential equations (LPDEs) of evolution type with variable coefficients, including parabolic equations, Schr\"odinger equation, and some other. Chernoff expressions are explicit formulas containing variable coefficients of LPDE and the initial condition, hence they can be used as a numerical method for solving LPDEs. However, the speed of convergence of such approximations at the present time is understudied which makes it risky to employ this class of numerical methods. 

In the present paper we take two equations with known solutions (heat equation and transport equation) and study both analytically and numerically the speed of decay of the norm of the difference between Chernoff approximations and exact solutions. We also provide graphical illustrations of convergence and its rate. These model examples, being relatively simple, allow to demonstrate general properties of Chernoff approximations. The observations obtained build a base for the future employment of the approach based on Chernoff's theorem to the problem of construction of new numerical methods for solving initial-value problem for parabolic LPDEs with variable coefficients.

\vskip2mm
\textbf{Keywords:} parabolic partial differential equation, Cauchy problem solution,  Chernoff approximations, convergence speed, estimates of convergence rate, numerical experiment

\vskip2mm
\textbf{MSC2010:} 35K15, 47D06, 65M12 
\small
\tableofcontents
\normalsize
\section{Introduction}
The Chernoff theorem \cite{chernoff_jfa_2008} is an effective tool of functional analysis which allows to construct approximations to $C_0$-semigroups that, in turn, provide solutions to some evolution equations (e.g., heat equation, parabolic equation with variable coefficients, Schr\"odinger equation, etc.); standard textbooks on the topic are \cite{engel_nagel_1999, pazy_2012, hille_1996}. If one finds so-called operator-valued Chernoff function then Chernoff's theorem generates a sequence of functions, also referred to as Chernoff approximations, that converges to the solution of a Cauchy problem for the evolution equation. In many cases constructing Chernoff function is the only method that allows to express the solution in terms of both variable coefficients of the equation and the chosen initial condition.

At the time the subject area has a decent amount of data on different methods to construct Chernoff functions. For instance, the members of Smolyanov’s group implemented the use of integral operators to find the solutions to parabolic equations in a plethora of cases during the last 15 years (refer to pioneering papers \cite{smolyanov_eeapilse_2003, smolyanov_jmp_2002}, overview \cite{smolyanov_tsa_2009, butko_2020} and some other results \cite{smolyanov_psim_2009, dubravina_jmp_2014, remizov_vedenin_2020}). The solutions obtained in the above mentioned studies were represented in the form of Feynman formula, i.e., as a limit of multiple integral when multiplicity goes to infinity. In turn, Schr\"odinger equation also belongs to the class of evolution equations, and the use of the same technique allows to represent the solution in the form of Feynman and quasi-Feynman integral formulas as well, see \cite{remizov_jfa_2016, remizov_doklady_2017}.

When dealing with expressions provided by the Chernoff theorem, one might be interested in obtaining the estimate for the rate of decay of approximation error as $n$ tends to infinity. Some results regarding this topic were already proposed for the case of Schr\"odinger equation in \cite{OSS2012}. Similarly, the question of convergence speed arises for the Trotter product formula $e^{A+B} = \lim_{n\to\infty}(e^{A/n}e^{B/n})^n$ which we do not exploit in current paper but which is a particular exemplar of Chernoff's theorem for $G(t) = e^{tA}e^{tB}$ (see recent papers \cite{zagrebnov_1, zagrebnov_2, zagrebnov_2020}). In a less specific setting, other than within the framework of Chernoff's theorem, one may refer to the results presented in \cite{GomTom2019, Gom2019, Gom2014}. 

Nevertheless, this area is not well studied at the present time in general. Starting from \cite{remizov_arxiv_2020} members of I.D.~Remizov's group examine the problem systematically. So, this paper is dedicated to the analysis of the convergence speed for some set of initial conditions and Chernoff functions. We suggest few analytic and computational methods in order to obtain the information on the order of approximation subspaces to which the examined model examples belong to as well as demonstrate the dependence of the order of the approximation on the choice of the initial condition. It is also shown on the model example that if the initial condition does not belong to the domain of the generator of the corresponding $C_0$-semigroup then the convergence speed can be slower than the one attained on the functions from the domain. 

Conducting this study implies the use of basic methods of infinite-dimensional functional analysis and operator semigroups, as well as the use of Matlab software under student license in order to perform numerical experiments and visualise the results.

\section{Preliminaries}

\subsection{$C_0$-semigroups and Chernoff's theorem}
This section is solely dedicated to the introduction of basic concepts that will be used throughout the rest of the paper. During the paper we will work with the system of equations given by

\begin{definition}
\label{def:cauchy_gen}
Let $X$ be an infinite set, and $\mathcal{F}$ be a Banach space of all number-valued functions on $X$. Moreover, assume having $dom(L)\subset\mathcal{F}$ and a closed linear operator $L\colon dom(L)\to\mathcal{F}$. Then the system of equations
\begin{equation}
\label{eq:cauchy_gen}
\begin{cases}
u'_t(t,x)=Lu(t,x),\\
u(0,x)=u_0(x),
\end{cases}
\end{equation}
where $x\in X$, $u_0\in\mathcal{F}$, and $u(t,\cdot)\in\mathcal{F}$ for all $t\geq0$, is called a Cauchy problem for the evolution equation.
\end{definition}

\begin{remark}
In case of existence of the $C_0$-semigroup $(G(t))_{t\geq0}$ with the generator $(L, dom(L))$ the solution of \eqref{eq:cauchy_gen} exists and is given by the equality $u(t,x)=(e^{tL}u_0)(x)$ for all $t\geq0$ and $x\in X$.
\end{remark}

\begin{definition}
If $u_0\in dom(L)$ then $u(t,\cdot)\in dom(L)$ for all $t\geq0$ and $u$ is called a classical solution.
\end{definition}

\begin{remark}
For an arbitrary $u_0\in\mathcal{F}$ the solution of \eqref{eq:cauchy_gen} exists only as a solution of the corresponding integral equation $u(t,\cdot)=L\int_{0}^{t}u(s,\cdot)ds+u_0$.
\end{remark}

After defining a Cauchy problem, it appears necessary to provide a common definition of a $C_0$-semigroup \cite{engel_nagel_1999}.
\begin{definition}
\label{def:semigroup}
Let $\mathcal{F}$ be a Banach space, and $\mathscr{L}(\mathcal{F})$ be a space of all linear bounded operators in $\mathcal{F}$. Let a mapping $V\colon [0, +\infty) \to \mathscr{L}(\mathcal{F})$ also be given, i.e., for a fixed $t\geq 0$ we have $V(t)$ being a linear bounded operator. Then $V$ is called a $C_0$-semigroup, or the same as strongly continuous one-parametric semigroup of linear bounded  operators, if the following conditions hold:
\begin{enumerate}
    \item $V(0) = I$, i.e., for all $\varphi\in\mathcal{F}$ we have $V(0)\varphi = \varphi$;
    \item under the action of $V$ the addition in $[0, +\infty)$ is mapped to the composition in $\mathscr{L}(\mathcal{F})$, i.e., for all $t,s\geq0$ we have $V(t+s)=V(t)\circ V(s)$, where $\circ$ denotes the composition of operators such that $(A\circ B)(\varphi) = A(B(\varphi))$ for all $\varphi\in\mathcal{F}$;
    \item $V$ is continuous in $\mathscr{L}(\mathcal{F})$ endowed with the strong operator topology, i.e., for all $\varphi\in\mathcal{F}$ a mapping $t\mapsto V(t)\varphi$ is continuous on $[0, +\infty)$.
\end{enumerate}
\end{definition}

The concept of a $C_0$-semigroup is closely related to the notion of its generator \cite{engel_nagel_1999}. So, we propose the following
\begin{definition}
\label{def:generator}
Let $(V(t))_{t\geq0}$ be a $C_0$-semigroup on a Banach space $\mathcal{F}$. Then the linear operator $L$ defined as
\begin{equation}
\label{eq:generator}
    L\varphi = \lim_{t\to +0}\frac{V(t)\varphi - \varphi}{t}
\end{equation}
on a linear subspace
\begin{equation}
\label{eq:domain}
    dom(L) = \left\{\varphi\in\mathcal{F} \;|\: \exists \lim_{t\to +0}\frac{V(t)\varphi - \varphi}{t}\right\} \subset \mathcal{F}
\end{equation}
is called an infinitesimal generator of a $C_0$-semigroup $(V(t))_{t\geq0}$. We also say that operator $L\colon dom(L) \to \mathcal{F}$ generates a $C_0$-semigroup and denote $V(t)=e^{tL}$.
\end{definition}

\begin{definition}
\label{def:closure}
Let $dom(L) \subset \mathcal{F}$, and $L\colon dom(L) \to \mathcal{F}$ be a linear operator. Then $L$ is called
\begin{itemize}
    \item \textit{closed} if its graph $\Gamma_{L} = \{(x, Lx) \in \mathcal{F}\times\mathcal{F} \;|\; x \in dom(L)\}$ is a closed subset in space $\mathcal{F}\times\mathcal{F}$ endowed with the norm $\|\cdot\|_{\mathcal{F}\times\mathcal{F}}$ given by $\|(x, y)\|_{\mathcal{F}\times\mathcal{F}} = \|x\|_{\mathcal{F}} + \|y\|_{\mathcal{F}}$ for all $x, y\in\mathcal{F}$;
    \item \textit{closable} (or admitting closure) if the closure of its graph $\overline{\Gamma_{L}}$ in space $\mathcal{F}\times\mathcal{F}$ with the norm $\|\cdot\|_{\mathcal{F}\times\mathcal{F}}$ is a graph of some operator $(\overline{L}, dom(\overline{L}))$, meaning $\Gamma_{\overline{L}} = \overline{\Gamma_{L}}$.
\end{itemize} 
\end{definition}

\begin{remark}
If operator $\overline{L}$ exists then it is a linear closed operator extending operator $L$, i.e., $dom(L) \subset dom(\overline{L})$, and $L|_{dom(L)} = \overline{L}|_{dom(L)}$.
\end{remark}

The next set of definitions is provided in the wording of I.D.~Remizov \cite{remizov_amc_2018}.
\begin{definition}
\label{def:chernoff_tan}
Let $\mathcal{F}$ be a Banach space, and $\mathscr{L}(\mathcal{F})$ be a space of all linear bounded operators in $\mathcal{F}$. Let a mapping $G\colon [0, +\infty) \to \mathscr{L}(\mathcal{F})$, or the same as a family of linear bounded operators $(G(t))_{t\geq0}$ in $\mathcal{F}$, also be given. Moreover, assume $dom(L)$ is dense in $\mathcal{F}$ and a linear operator $L\colon dom(L) \to \mathcal{F}$ is closed. Then function $G$ is called \textit{Chernoff tangent} to operator $L$ if the following conditions are satisfied:
\begin{enumerate}[(CT1).]
    \item function $G$ is strongly continuous, or the same as continuous in $\mathscr{L}(\mathcal{F})$ endowed with the strong operator topology, i.e., the mapping $t\mapsto G(t)f$ is continuous on $[0, +\infty)$ for all $f\in\mathcal{F}$;
    \item $G(0) = I$, i.e., for all $f\in\mathcal{F}$ we have $G(0)f = f$;
    \item there exists a dense linear subspace $D\subset\mathcal{F}$ such that for all $f\in D$ we have 
    $$\lim_{t\to +0}(G(t)\varphi - \varphi)/t \overset{denote}{=} G'(0)f\in\mathcal{F};$$
    \item the closure of operator $(G'(0), D)$ exists and is equal to $(L, dom(L))$.
\end{enumerate}
\end{definition}

\begin{remark}
Conditions (CT3) and (CT4) together mean that there exists a core $D$ for $L$.
\end{remark}

Classical wording of the next theorem can be found in \cite{chernoff_jfa_2008, engel_nagel_1999}. Based on the notion of Chernoff tangency we provide
\begin{theorem}[The Chernoff theorem]
\label{th:chernoff}
Let $\mathcal{F}$ be a Banach space, and $\mathscr{L}(\mathcal{F})$ be a space of all linear bounded operators in $\mathcal{F}$. Let a mapping $G\colon [0, +\infty) \to \mathscr{L}(\mathcal{F})$, or the same as a family of linear bounded operators $(G(t))_{t\geq0}$ in $\mathcal{F}$, also be given. Moreover, assume having a closed linear operator $L\colon dom(L) \subset \mathcal{F} \to \mathcal{F}$. Suppose that the following conditions are met:
\begin{itemize}
    \item [(E).] there exists a $C_0$-semigroup $(e^{tL})_{t\geq 0}$ with generator $(L, dom(L))$;
    \item [(CT).] function $G$ is Chernoff tangent to operator $L$;
    \item [(N).] there exists number $\omega\in\mathbb{R}$ such that $\|G(t)\|\leq e^{\omega t}$ for all $t\geq 0$.
\end{itemize}
Then for each $f\in\mathcal{F}$ we have $(G(t/n))^nf\xrightarrow{n\to\infty}e^{tL}f$ uniformly in $t\in[0, T]$ for any fixed $T\geq0$, i.e., for each $f\in\mathcal{F}$ and each $T\geq 0$ we get
\begin{equation}
\label{eq:chernoff}
    \lim_{n\to\infty}\sup_{t\in[0, T]}\left\|\left(G\left(\frac{t}{n}\right)\right)^{n}f-e^{tL}f\right\|=0.    
\end{equation}
\end{theorem}
\begin{definition}
If the above conditions hold then $G$ is called a Chernoff function for $L$.
\end{definition}

\subsection{Conjectures on convergence speed for Chernoff approximations}
During the analysis of the rate of convergence both for the transport and heat equations, our aim is to check the validity of the following conjectures for several model examples of initial conditions. 

\begin{conjecture}[I. Remizov, 2018]
\label{con:1}
Let $(e^{tL})_{t\geq 0}$ be a $C_0$-semigroup in a Banach space $\mathcal{F}$ with generator $(L, dom(L))$, and $G$ be a Chernoff function for operator $L$. Moreover, assume $t_0\geq0$ and $f\in\mathcal{F}$ are given, and suppose that for all $t\in[0, t_0]$ we have $f\in dom(G'(t))\cap dom(G''(t))$ and functions $t\to G'(t)f$ and $t\to G''(t)f$ are continuous. Then there exists a number $C\geq0$ such that for all $t\in[0, t_0)$ and all $n\in\mathbb{N}$ the estimate $\left\|\left(G\left(\frac{t}{n}\right)\right)^{n}f-e^{tL}f\right\|\leq \frac{C}{n}$ holds.
\end{conjecture}

\begin{conjecture}[I. Remizov, 2018]
\label{con:2}
Let $(e^{tL})_{t\geq 0}$ be a $C_0$-semigroup in a Banach space $\mathcal{F}$ with generator $(L, dom(L))$, and $G$ be a Chernoff function for operator $L$. Moreover, assume number $t_0\geq0$ and $f\in\mathcal{F}$ are given, and suppose that for all $t\in[0, t_0]$ we have function $f$ being from the intersection of the domains of $G'(t)$, $G''(t)$, $G'''(t)$, $G''''(t)$, $G'(t)G''(t)$, $G'(t)^{2}G''(t)$, $G''(t)^{2}$, having that each of these operators is continuous for all $t\in[0, t_0]$. Then there exists a number $C\geq0$ such that for all $t\in[0, t_0)$ and all $n\in\mathbb{N}$ the following inequality is valid:
$$\left\|\left(G\left(\frac{t}{n}\right)\right)^{n}f-e^{tL}f+\frac{t^2}{2n}e^{tL}(L^2-G''(0))f\right\|\leq \frac{C}{n^2}.$$
\end{conjecture}

For the further convenience of the interpretation of numerical studies we provide the definition of an approximation subspace proposed in \cite{remizov_arxiv_2020}.
\begin{definition}
Let $\tau \subset [0, +\infty) = \mathbb{R}^{+}$, and map $\psi \colon \mathbb{R}^{+}\to\mathbb{R}^{+}$ be such that $\lim_{x \to +\infty}\psi(x) = 0$. Then the set $A^{\tau}_{\psi} = \{f \in \mathcal{F}\;|\; \sup_{t\in\tau}\left\|\left(G\left(\frac{t}{n}\right)\right)^{n}f-e^{tL}f\right\| = O(\psi(n)) \;\mbox{as}\; n \to \infty\}$ is called an approximation subspace of order $\psi$.
\end{definition}

\begin{remark}
The approximation subspace $A^{\tau}_{\psi}$ is indeed a linear subspace of $\mathcal{F}$.
\end{remark}

Hence, we now state the previous Conjecture 2 in the new wording in
\begin{conjecture}
Let $(e^{tL})_{t\geq 0}$ be a $C_0$-semigroup in a Banach space $\mathcal{F}$ with generator $(L, dom(L))$, and $G$ be a Chernoff function for operator $L$. More than that, assume number $t_0\geq0$ is given, and denote the intersection of the domains of $G'(t)$, $G''(t)$, $G'''(t)$, $G''''(t)$, $G'(t)G''(t)$, $G'(t)^{2}G''(t)$, $G''(t)^{2}$ for all $t\in[0, t_0]$ as $D(t_0)$. Assume that $D(t_0)$ is dense in $\mathcal{F}$ and each of these operators are continuous in $t$ on all vectors from $D(t_0)$ for each $t\in[0, t_0]$. Then function $G$ has a dense approximation subspace $A^{[0, t_0)}_{\zeta}$ of order $\zeta(n) = \frac{1}{n^2}$, and this subspace is a subset of $D(t_0)$.
\end{conjecture}

Despite the fact that only fast converging approximations have a practical value for the research as a whole, it is worth mentioning that Chernoff functions can be likewise constructed in a way to provide an arbitrary slow convergence rate. This result is presented in the next

\begin{proposition}[I. Remizov et al., 2019]
\label{prop:1}
The convergence speed in the Chernoff theorem can be arbitrary slow. That is, if a function $w\colon[0,+\infty]\to[0,+\infty]$ such that $\lim_{x\to+\infty}w(x)=0$ is given then there exist such $X,\mathcal{F}, L, e^{tL}, G, f$ for which the equality $w(n/t)=O(\left\|\left(G\left(\frac{t}{n}\right)\right)^{n}f-e^{tL}f\right\|)$ holds for all $t\geq0$ when $n\to\infty$.
\end{proposition}

\section{Exact formulas and estimates for convergence speed}
During this section we will propose the results on the approximation speed for some model initial conditions and both one-dimensional transport and heat equations.

\subsection{Transport equation}
As a preliminary we will commence with the notion of the translation semigroup that gives the solution for one-dimensional transport equation. The examined system of equations is stated in
\begin{definition}
The system of equations defined as
\begin{equation}
\label{eq:cauchy_tr}
\begin{cases}
u'_t(t,x)=u'_{x}(t,x),\\
u(0,x)=u_0(x),
\end{cases}
\end{equation}
where $x\in \mathbb{R}$, $u_0\in UC_b(\mathbb{R})$ for all $t\geq0$, is called a Cauchy problem for the first order partial differential equation on the real line, also referred to as the transport equation.
\end{definition}

\begin{remark}
The solution of the system \eqref{eq:cauchy_tr} and, hence, the corresponding $C_0$-semigroup is given by the formula $u(t,x)=(e^{tL}u_0)(x)=u_0(x+t)$. This $C_0$-semigroup is called the translation semigroup.
\end{remark}

First, let us verify that the generator of the translation semigroup is indeed the differentiation operator.
\begin{theorem}
\label{th:generator}
Let $(G(t))_{t\geq 0}$ be the $C_0$-semigroup given by the formula $(G(t)f)(x)=f(x+t)$ for $f\in UC_b(\mathbb{R})$ and $t\geq0$. Then its generator is differentiation operator $L=\frac{d}{dx}$, and its domain $dom(L)$ is
\begin{equation*}
    UC^{1}_b(\mathbb{R}) \overset{denote}{=} \left\{f\in UC_b(\mathbb{R}) \;|\; f'\in UC_b(\mathbb{R})\right\}.
\end{equation*}
\end{theorem}
\begin{proof}
Using Definition \ref{def:generator} of the generator of a $C_0$-semigroup we get
\begin{equation}
\label{eq:gen_limit}
    dom(L)=\left\{f\in UC_b(\mathbb{R}) \;|\; \lim_{t\to +0}\left\|\frac{G(t)f-f}{t}-Lf\right\|=0\right\},
\end{equation}
where $\|\cdot\|$ denotes the standard supremum norm in $UC_b(\mathbb{R})$. Applying the definition of the $C_0$-semigroup, we obtain
\begin{equation}
\label{eq:uni_diff}
    \lim_{t\to+0}\sup_{x\in\mathbb{R}}\left|\frac{f(x+t)-f(x)}{t}-f'(x)\right| = 0.    
\end{equation}
Now, let us show that $f'$ is uniformly continuous. Take a sequence $(t_n)_{n\in\mathbb{N}}$ such that $\lim_{n\to\infty}t_n=0$ and $t_n\geq0$ for all $n\in\mathbb{N}$. Then it follows that $\frac{f(x+t_n)-f(x)}{t_n}\xrightarrow{n\to\infty}f'(x)$ uniformly in $UC_b(\mathbb{R})$ norm for some $x\in\mathbb{R}$. Since functions given by the formula $\frac{f(x+t_n)-f(x)}{t_n}$ are uniformly continuous for all $n\in\mathbb{N}$, by Weierstrass theorem $f'$ is also uniformly continuous. 

Next, we will show that $f'$ is bounded. By the definition of uniform continuity we have that 
\begin{equation*}
    \forall \varepsilon > 0 \; \exists\, N = N(\varepsilon) > 0 \:\colon \forall n \geq N \; \forall x\in\mathbb{R} \implies \sup_{x \in \mathbb{R}}\left|\frac{f(x+t_n)-f(x)}{t_n}-f'(x)\right|<\varepsilon.
\end{equation*}
So we get $\left|f'(x)\right|<\left|\frac{f(x+t_n)-f(x)}{t_n}\right|+\varepsilon$. Take $\varepsilon_0=1$ then $f'$ is bounded. 

Finally, let us show that if $f'\in UC_b(\mathbb{R})$ then $\lim_{t\to-0}\sup_{x\in\mathbb{R}}\left|\frac{f(x+t)-f(x)}{t}-f'(x)\right| = 0$. For all $t<0$ we have
\begin{equation*}
\begin{split}
    \left|\frac{f(x+t)-f(x)}{t}-f'(x)\right| & = \left|\frac{f(x-|t|)-f(x)}{-|t|}-f'(x)\right| \\
    & = \left|\frac{f(x)-f(x-|t|)}{|t|}-f'(x)\right| \\
    &\leq \left|\frac{f((x-|t|)+|t|)-f(x-|t|)}{|t|}-f'(x-|t|)\right|+\left|f'(x-|t|)-f'(x)\right|.
\end{split}    
\end{equation*}
For small $|t|$ the former term is small due to \eqref{eq:uni_diff}, and the latter one due to $f'$ being uniformly continuous. So, we get that
\begin{equation*}
    \lim_{t\to+0}\frac{f(x+t)-f(x)}{t} = f'(x) = \lim_{t\to-0}\frac{f(x+t)-f(x)}{t}.
\end{equation*}
Hence, if the limit given by \eqref{eq:gen_limit} exists, meaning $f\in dom(L)$, then $f\in UC^{1}_b(\mathbb{R})$. On the other hand, if $f\in UC^{1}_b(\mathbb{R})$ then by mean value theorem
\begin{equation*}
    \left|\frac{f(x+t_n)-f(x)}{t_n}-f'(x)\right| = \left|f'(x+\theta)-f'(x)\right| < \varepsilon \;\mbox{for}\; |t| < \delta \;\mbox{and}\; 0<\theta<t_n,
\end{equation*}
i.e., $f \in dom(L)$. So, we proved that $L$ is generator of this semigroup and $dom(L)=UC^{1}_b(\mathbb{R})$.
\end{proof}

Prior to the numerical studies we will find the composition degree of some Chernoff functions $G$ for an arbitrary initial condition $u_0\in UC_b(\mathbb{R})$.
\begin{proposition}
\label{prop:comp_degree_tr}
Let $G$ be the Chernoff function such that for all $f\in UC_b(\mathbb{R})$ we have
\begin{equation}
\label{eq:tr_apprx}
    (G(t)f)(x)=f(x+t+at^{k+1})
\end{equation}
for some fixed $a, k > 0$. Then its $n$-th composition degree is given by the formula
\begin{equation}
\label{eq:comp_degree_tr}
    \left(\left(G\left(\frac{t}{n}\right)\right)^{n}u_0\right)(x)=u_0\left(x+t+\frac{at^{k+1}}{n^k}\right).
\end{equation}
\end{proposition}

\begin{proof}
By the definition of Chernoff function we get
\begin{equation*}
    \left(\left(G\left(\frac{t}{n}\right)\right)^{n}u_0\right)(x) = u_0\left(x+n\left(\frac{t}{n}+a\left(\frac{t}{n}\right)^{k+1}\right)\right) = u_0\left(x+t+\frac{at^{k+1}}{n^k}\right)
\end{equation*}
for all $x\in\mathbb{R}$ and $t \geq 0$.
\end{proof}

\subsubsection{Analysis of convergence speed for \boldmath{$u_0 = [x\to\sin(x)]$}}
Now, let us check the convergence speed for $u_0=[x\mapsto\sin(x)]$ and Chernoff function \eqref{eq:tr_apprx}, the prototype of which was proposed in \cite{remizov_arxiv_2020}. The main result of this paragraph is presented in the following
\begin{theorem}
\label{th:sin_tr}
Let $(e^{tL})_{t\geq0}$ be the translation semigroup on the real line, and $G$ be the Chernoff function \eqref{eq:tr_apprx} for some fixed $a, k > 0$. Then function $u_0 = [x\mapsto\sin(x)]$ belongs to the approximation subspace of order $\frac{1}{n^{k}}$, and the error is given by the formula \begin{equation*}
    \left\|\left(G\left(\frac{t}{n}\right)\right)^{n}u_0-e^{tL}u_0\right\| = 2\left|\sin(at^{k+1}/2n^{k})\right|.
\end{equation*}
\end{theorem}
\begin{proof}

Denote $\xi(n) = \left\|\left(G\left(\frac{t}{n}\right)\right)^{n}u_0-e^{tL}u_0\right\|$ and $\tau = at^{k+1}/n^{k}$. Applying the definition of a supremum norm in $UC_b(\mathbb{R})$ and using \eqref{eq:chernoff}, we get
\begin{equation*}
\begin{split}
    \xi(n) & = \sup_{x\in\mathbb{R}}\left|\sin(x+t+\tau)-\sin(x+t)\right| \\
    & = \sup_{x\in\mathbb{R}}\left|(\sin(x)\cos(t+\tau)+\cos(x)\sin(t+\tau))-(\sin(x)\cos(t)+\cos(x)\sin(t))\right| \\
    & = \sup_{x\in\mathbb{R}}\left|(\cos(t+\tau)-\cos(t))\sin(x) + (\sin(t+\tau)-\sin(t))\cos(x)\right|. \\
\end{split}
\end{equation*}
The above supremum can be found based on the following assumption: as $\sin$ and $\cos$ are bounded functions then the value of the supremum is reached in an extremum of the internal function. Let us now find an extreme point. We have that
\begin{equation*}
    (\cos(t+\tau)-\cos(t))\cos(x) - (\sin(t+\tau)-\sin(t))\sin(x) = 0, 
\end{equation*}
and
\begin{equation*}
    \tan(x)=\frac{\cos(t+\tau)-\cos(t)}{\sin(t+\tau)-\sin(t)}=\frac{-2\sin(t+\tau)\sin(\tau)}{2\cos(t+\tau)\sin(\tau)}=\tan(-(t+\tau)),
\end{equation*}
thus meaning that the extremum is reached at the point $x=-(t+\tau)$ regardless of the period of $\tan$. So, we get
\begin{equation*}
\begin{split}
    \xi(n) & = \left|(\sin(t+\tau)-\sin(t))\cos(t+\tau)-(\cos(t+\tau)-\cos(t))\sin(t+\tau)\right| \\
    & = \left|2\cos(t+\tau)\cos(t+\tau)\sin(\tau)+2\sin(t+\tau)\sin(t+\tau)\sin(\tau)\right| \\
    & = 2\left|(\cos^2(t+\tau)+\sin^2(t+\tau))\sin(\tau)\right| \\
    & = 2\left|\sin(\tau)\right|. \\
\end{split}
\end{equation*}
Now, using Taylor expansion $\sin(z)=z-\frac{z^3}{3!}+O(z^5)$ when $z \to 0$, we finally get 
\begin{equation*}
    \left\|\left(G\left(\frac{t}{n}\right)\right)^{n}u_0-e^{tL}u_0\right\| = 2\sin(\tau)=\frac{at^{k+1}}{n^{k}}+O\left(\frac{1}{n^{k+1}}\right) \; \mbox{when} \; n\to \infty,
\end{equation*}
which proves that $u_0 = [x\mapsto\sin(x)]$ belongs to an approximation subspace of order $\frac{1}{n^k}$.
\end{proof}

\subsubsection{Analysis of slow converging approximations for \boldmath{$u_0 = [x\to\sin(x)]$}}
At last, we will generalize Theorem \ref{th:sin_tr} for an arbitrary function $w\colon\mathbb{R}^{+}\to\mathbb{R}^{+}$.
\begin{theorem}
\label{th:sin_slow_tr}
Let $(e^{tL})_{t\geq0}$ be the translation semigroup on the real line, and $G$ be the Chernoff function such that for all $f\in UC_b(\mathbb{R})$ we have 
\begin{equation}
\label{eq:tr_apprx_slow}
    (G(t)f)(x)=f(x+t+t\cdot w(1/t)),
\end{equation}
where $w\colon\mathbb{R}^{+}\to\mathbb{R}^{+}$ such that $\lim_{n\to\infty} w(n) = 0$. Then function $u_0 =[x \mapsto \sin(x)]$ belongs to the approximation subspace of order $w$, and the error is given by the formula 
\begin{equation*}
    \left\|\left(G\left(\frac{t}{n}\right)\right)^{n}u_0-e^{tL}u_0\right\| = 2\left|\sin(t/2\cdot w(n/t))\right|.
\end{equation*}
\end{theorem}
\begin{proof}
Similarly to Proposition \ref{prop:comp_degree_tr} we first obtain $n$-th composition degree for Chernoff function $G$. We get
\begin{equation}
\label{eq:5}
    \left(\left(G\left(\frac{t}{n}\right)\right)^{n}u_0\right)(x)=u_0\left(x+n\left(\frac{t}{n} + \frac{t}{n}\cdot w(n/t)\right)\right)=u_0\left(x+t+w(n/t)t\right).
\end{equation}
It now suffices to set $\tau = t\cdot w(n/t)$ and use the same technique as of Theorem \ref{th:sin_tr}. Hence, since $w(n) \xrightarrow{n\to \infty} 0$, we have
\begin{equation*}
    \xi(n) = 2\left|\sin(t/2\cdot w(n/t))\right|=\left|t\cdot w(n/t)+o\left(w^{2}(n)\right)\right| \;\mbox{as}\; n \to \infty,
\end{equation*}
which proves that $u_0 =[x \mapsto \sin(x)]$ belongs to an approximation subspace of order $w$.
\end{proof}

Later in Section 4 we will check the convergence speed attained for the  slow converging Chernoff function $G$ given by \eqref{eq:tr_apprx_slow} with $w(t)=\frac{1}{t^{\gamma}}$ for different values of $0 < \gamma < 1$ applied to approximate the solution for $u_0 =[x \mapsto \sin(x)]$.

\subsection{Heat equation}

The next differential equation we are going to analyse is the heat equation. Now, we will state the corresponding Cauchy problem and give an explicit formula for the solution in a form of a $C_0$-semigroup.

\begin{definition}
\label{def:cauchy_heat}
The system of equations defined the following way
\begin{equation}
\label{eq:cauchy_heat}
\begin{cases}
u'_t(t,x)=Lu(t,x)=a^2u''_{xx}(t,x),\\
u(0,x)=u_0(x),
\end{cases}
\end{equation}
where $a>0$, $x\in \mathbb{R}$, and $u_0\in UC_b(\mathbb{R})$ for all $t\geq0$, is called a Cauchy problem for the heat equation on the real line with constant coefficient of thermal conductivity.
\end{definition}

\begin{remark}
The solution of the system \eqref{eq:cauchy_heat}, and, hence, the corresponding $C_0$-semigroup, is given by the Poisson integral 
\begin{equation}
\label{eq:poisson}
u(t,x)=(e^{tL}u_0)(x)=\int_{\mathbb{R}}\Phi(t, x - y)u_0(y)dy,
\end{equation}
where 
\begin{equation}
\label{eq:poisson_integral}
\Phi(t, x) = \frac{1}{2a\sqrt{\pi t}}\exp\left(-\frac{x^2}{4a^2t}\right).
\end{equation}
\end{remark}

\bigskip
We will analyse the performance of approximations for the following three Chernoff functions:
\begin{equation}
\label{eq:chernoff_heat_first}
    \left(G(t)u_0\right)(x) = \frac{1}{4}u_0(x+2a\sqrt{t}) + \frac{1}{4}u_0(x-2a\sqrt{t})+ \frac{1}{2}u_0(x)
\end{equation}
proposed by I.~Remizov in \cite{remizov_amc_2018},
\begin{equation}
\label{eq:chernoff_heat_second}
    (G(t)u_0)(x)=\frac{1}{6}u_0(x+a\sqrt{6t})+\frac{1}{6}u_0(x-a\sqrt{6t})+\frac{2}{3}u_0(x)
\end{equation}
proposed by A.~Vedenin in \cite{remizov_arxiv_2020}, and
\begin{equation}
\label{eq:chernoff_heat_third}
    (G(t)u_0)(x)=\frac{1}{30}u_0(x+a\sqrt{12t})+\frac{1}{30}u_0(x-a\sqrt{12t})+\frac{3}{10}u_0(x+a\sqrt{2t})+\frac{3}{10}u_0(x-a\sqrt{2t})+\frac{1}{3}u_0(x)
\end{equation}
also proposed by A.~Vedenin in oral communication to the author of the present paper in 2020. To simplify the complexity of computational algorithms one might need a general formula for the $n$-th composition degree for the above Chernoff functions. So we present the following
\begin{proposition}
\label{prop:chernoff_heat_first_binomial}
Let $G$ be the Chernoff function given by \eqref{eq:chernoff_heat_first} for all $f\in UC_b(\mathbb{R})$ and fixed $a>0$. Then its $n$-th composition degree is given by the formula
\begin{equation}
\label{eq:comp_degree_heat_first}
    \left(\left(G\left(\frac{t}{n}\right)\right)^{n}u_0\right)(x) = \frac{1}{4^n}\sum^{n}_{p = -n}\alpha_{p, n}u_{0}(x+2ap\sqrt{t/n}),
\end{equation} 
where for each $n\in\mathbb{N}$ and $p \in \{-n, \dots, n\}$ coefficients $\alpha_{p, n}$ are defined as
\begin{equation*}
    \alpha_{p, n} = \sum^{\left[\frac{n-|p|}{2}\right]}_{k = 0}C^{n}_{k}C^{k+|p|}_{n - k}2^{n-|p|-2k}.
\end{equation*}
\end{proposition}

\begin{proof}
Consider $G(t) = \frac{1}{4}(S_{2a\sqrt{t}} + S_{-2a\sqrt{t}} +2I)$ for all $t > 0$, where $I$ is an identity operator and $S_{\tau}$ is a shift operator given by $(S_{\tau}f)(x) = f(x + \tau)$ for some $\tau\in\mathbb{R}$. It then yields that 
\begin{equation*}
\begin{split}
    G(t)^{n}_{t/n} & = \frac{1}{4^n}(S_{2a\sqrt{t/n}} + S_{-2a\sqrt{t/n}} +2I)^n \\
    & = \frac{1}{4^n}\sum_{0\leq k + m\leq n}S^{k}_{2a\sqrt{t/n}}S^{m}_{-2a\sqrt{t/n}}2^{n-k-m}\frac{n!}{k!m!(n-k-m)!} \\
    & = \frac{1}{4^n}\sum_{0\leq k + m\leq n}S_{2a(k-m)\sqrt{t/n}}2^{n-k-m}\frac{n!}{k!m!(n-k-m)!}.
\end{split}
\end{equation*}
Now, set $k - m = p$ with $-n\leq p\leq n$. We obtain
\begin{equation*}
\begin{split}
    G(t)^{n}_{t/n} & = \frac{1}{4^n}\sum^{n}_{p = -n}\sum^{\left[\frac{n-|p|}{2}\right]}_{k = 0}C^{k}_{n}C^{k+|p|}_{n-k}2^{n-|p|-2k}S_{2ap\sqrt{t/n}} \\
    & = \frac{1}{4^n}\sum^{n}_{p = -n}\alpha_{p, n}S_{2ap\sqrt{t/n}},
\end{split}
\end{equation*}
where $\alpha_{p, n} = \sum^{\left[\frac{n-|p|}{2}\right]}_{k = 0}C^{n}_{k}C^{k+|p|}_{n - k}2^{n-|p|-2k}$.
\end{proof}
Similarly, one may derive the next
\begin{proposition}
\label{prop:chernoff_heat_second_binomial}
Let $G$ be the Chernoff function given by \eqref{eq:chernoff_heat_second} for all $f\in UC_b(\mathbb{R})$ and fixed $a>0$. Then its $n$-th composition degree is given by the formula
\begin{equation}
\label{eq:comp_degree_heat_second}
    \left(\left(G\left(\frac{t}{n}\right)\right)^{n}u_0\right)(x) = \frac{1}{6^n}\sum^{n}_{p = -n}\beta_{p, n}u_{0}(x+ap\sqrt{6t/n}),
\end{equation}
where for each $n\in\mathbb{N}$ and $p \in \{-n, \dots, n\}$ coefficients $\beta_{p, n}$ are defined as
\begin{equation*}
    \beta_{p, n} = \sum^{\left[\frac{n-|p|}{2}\right]}_{k = 0}C^{n}_{k}C^{k+|p|}_{n - k}4^{n-|p|-2k}.
\end{equation*}
\end{proposition}
\subsubsection{Analysis of convergence rate for the initial condition \boldmath{$u_0(x) = [x\to\sin(x)]$}}
Let us now find the solution of the above Cauchy problem \eqref{eq:cauchy_heat} for $u_0 =[x \mapsto \sin(x)]$.

\begin{theorem}
Suppose $u_0 =[x \mapsto \sin(x)]$. Then the solution of the system \eqref{eq:cauchy_heat} with initial condition $u_0$ is given by the formula $u(t,x)=(e^{tL}u_0)(x)=e^{-a^2t}\sin(x)$.
\end{theorem}
\begin{proof}
Using \eqref{eq:poisson} and substituting $u_0 =[x \mapsto \sin(x)]$, we get 
\begin{equation*}
    u(t,x)=(e^{tL}u_0)(x)=\int_{\mathbb{R}}\frac{1}{2a\sqrt{\pi t}}\exp\left(-\frac{(x-y)^2}{4a^2t}\right)\sin(y)dy.
\end{equation*}

Now, let us apply the following identity $\sin(x)=\frac{e^{ix}-e^{-ix}}{2i}$. So, we obtain
\begin{equation*}
    u(t, x) = \frac{1}{2i}\frac{1}{2a\sqrt{\pi t}}\left(\int_{\mathbb{R}}\exp\left(-\frac{(x-y)^2}{4a^2t}\right)\exp(iy)dy-\int_{\mathbb{R}}\exp\left(-\frac{(x-y)^2}{4a^2t}\right)\exp(-iy)dy\right). \\
\end{equation*}

The former integral now takes the form 
\begin{equation*}
\begin{split}
    \int_{\mathbb{R}}\frac{1}{2a\sqrt{\pi t}}\exp\left(-\frac{(y-(x+2a^2ti))^2-(x+2a^2ti)^2+x^2}{4a^2t}\right)dy & = \exp\left(-\frac{-(x+2a^2ti)^2+x^2}{4a^2t}\right) \\
    & = \exp\left(ix-a^2t\right).
\end{split}
\end{equation*}

Similarly, the latter integral is equal to $\exp\left(-ix-a^2t\right)$, i.e., the solution can now be expressed as
\begin{equation}
\label{eq:heat_semigroup_sin}
    u(t,x)=(e^{tL}u_0)(x) = e^{-a^2t}\left(\frac{e^{ix}-e^{-ix}}{2i}\right)=e^{-a^2t}\sin(x).
\end{equation}

So, we found the solution of the Cauchy problem given by \eqref{eq:cauchy_heat} for $u_0 =[x \mapsto \sin(x)]$.
\end{proof}

Prior to analysing the convergence rate for Chernoff function \eqref{eq:chernoff_heat_first} and $u_0 = [x\to\sin(x)]$ let us compute its $n$-th composition degree.
\begin{proposition}
\label{prop:sin_first_heat_simpl}
Let $G$ be the Chernoff function given by \eqref{eq:chernoff_heat_first} for all $f\in UC_b(\mathbb{R})$ and fixed $a>0$. Then its $n$-th composition degree for $u_0 =[x \mapsto \sin(x)]$ has the form
\begin{equation}
\label{eq:comp_degree_heat_sin_first}
    \left(\left(G\left(\frac{t}{n}\right)\right)^{n}u_0\right)(x) = \left(\cos\left(a\sqrt{\frac{t}{n}}\right)\right)^{2n}\sin(x).
\end{equation}
\end{proposition}

\begin{proof}
Substituting the initial condition into \eqref{eq:chernoff_heat_first}, we get 
\begin{equation*}
\begin{split}
    \left(G(t)u_0\right)(x) & = \frac{1}{4}\sin(x+2a\sqrt{t}) + \frac{1}{4}\sin(x-2a\sqrt{t})+ \frac{1}{2}\sin(x) \\
    & = \left(\frac{1 + \cos(2a\sqrt{t})}{2}\right)\sin(x) \\
    & = \cos^{2}(a\sqrt{t})\sin(x).
\end{split}
\end{equation*}
So, for all $t\geq0$ operator $G(t)$ is a function multiplication operator, meaning its $n$-th composition degree has the following form 
\begin{equation*}
    \left(\left(G\left(\frac{t}{n}\right)\right)^{n}u_0\right)(x) = \left(\cos\left(a\sqrt{\frac{t}{n}}\right)\right)^{2n}u_0(x) = \left(\cos\left(a\sqrt{\frac{t}{n}}\right)\right)^{2n}\sin(x)
\end{equation*}
for all $t \geq 0$ and $x\in\mathbb{R}$.
\end{proof}

Finally, let us find the convergence speed for $u_0 =[x \mapsto \sin(x)]$ and the above derived Chernoff function \eqref{eq:comp_degree_heat_sin_first}. The main result of this subsection is presented in the next 
\begin{theorem}
\label{th:sin_first_heat}
Let $(e^{tL})_{t\geq0}$ be the heat $C_0$-semigroup, and $G$ be the Chernoff function given by \eqref{eq:chernoff_heat_first} for all $f\in UC_b(\mathbb{R})$ and fixed $a>0$. Then function $u_0 =[x \mapsto \sin(x)]$ belongs to the approximation subspace of order $\frac{1}{n}$.
\end{theorem}
\begin{proof}
Applying the definition of a standard norm in $UC_b(\mathbb{R})$, we get

\begin{equation*}
\begin{split}
\left\|\left(G\left(\frac{t}{n}\right)\right)^{n}u_0-e^{tL}u_0\right\| & = \sup_{x\in\mathbb{R}}\left|\left(\cos\left(a\sqrt{\frac{t}{n}}\right)\right)^{2n}\sin(x)-e^{-a^2t}\sin(x)\right|\\
& = \left|\left(\cos\left(a\sqrt{\frac{t}{n}}\right)\right)^{2n}-e^{-a^2t}\right|\sup_{x\in\mathbb{R}}\left|\sin(x)\right|.
\end{split}
\end{equation*}
Now, let us exploit the following Taylor expansions:
\begin{equation*}
    \cos(z)=1-\frac{z^2}{2!}+O(z^4),
\end{equation*} 
and 
\begin{equation*}
    \ln(1+z)^\alpha=z-\frac{z^2}{2!}+O(z^4),
\end{equation*}
and
\begin{equation*}
    \exp(z) = 1 + z + \frac{z^2}{2!}+O(z^3)
\end{equation*}
when $z \to 0$. So, we have
\begin{equation*}
\begin{split}
    \left\|\left(G\left(\frac{t}{n}\right)\right)^{n}u_0-e^{tL}u_0\right\| & = \left|e^{2n\ln\left(\cos\left(a\sqrt{\frac{t}{n}}\right)\right)}-e^{-a^2t}\right| \\
    & = \left|-\frac{a^4t^2}{6n}+O\left(\frac{1}{n}\right)\right|, \;\mbox{when}\; n \to \infty.
\end{split}
\end{equation*}
This proves that $u_0 =[x \mapsto \sin(x)]$ belongs to the approximation subspace of order $\frac{1}{n}$.
\end{proof}

Similarly, applying the same technique as of Theorem \ref{th:sin_first_heat} one may propose the next series of results containing the estimates of the convergence rate for Chernoff functions \eqref{eq:chernoff_heat_second} and \eqref{eq:chernoff_heat_third}, starting with
\begin{proposition}
\label{prop:sin_second_heat_simpl}
Let $G$ be the Chernoff function given by \eqref{eq:chernoff_heat_second} (respectively \eqref{eq:chernoff_heat_third}) for all $f\in UC_b(\mathbb{R})$ and fixed $a>0$. Then its $n$-th composition degree for $u_0 =[x \mapsto \sin(x)]$ is given by
\begin{equation}
\label{eq:comp_degree_heat_sin_second}
    \left(\left(G\left(\frac{t}{n}\right)\right)^{n}u_0\right)(x) = \frac{1}{3^{n}}\left(2 + \cos\left(a\sqrt{\frac{6t}{n}}\right)\right)^{n}\sin(x)
\end{equation}
respectively
\begin{equation}
\label{eq:comp_degree_heat_sin_third}
    \left(\left(G\left(\frac{t}{n}\right)\right)^{n}u_0\right)(x) = \frac{1}{15^{n}}\left(5 + \cos\left(a\sqrt{\frac{12t}{n}}\right) + 9\cos\left(a\sqrt{\frac{2t}{n}}\right)\right)^{n}\sin(x).
\end{equation}
\end{proposition}
\begin{theorem}
\label{th:sin_second_heat}
Let $(e^{tL})_{t\geq0}$ be the heat $C_0$-semigroup, and $G$ be the Chernoff function given by \eqref{eq:chernoff_heat_second} (respectively \eqref{eq:chernoff_heat_third}) for all $f\in UC_b(\mathbb{R})$ and fixed $a>0$. Then function $u_0 =[x \mapsto \sin(x)]$ belongs to the approximation subspace of order $\frac{1}{n^2}$ and $\frac{1}{n^3}$ respectively.
\end{theorem}

\subsubsection{Finding the exact solution for the initial condition \boldmath{$u_0 = [x\to\exp(-|x|)]$}}
Here we will find the solution of the Cauchy problem given by \eqref{eq:cauchy_heat} for $u_0=[x\mapsto e^{-|x|}]$.

\begin{theorem}
Suppose $u_0=[x\mapsto e^{-|x|}]$. Then the solution of the system \eqref{eq:cauchy_heat} with initial condition $u_0$ is given by the formula 
\begin{equation}
\label{eq:heat_semigroup_exp}
    u(t,x)=(e^{tL}u_0)(x)=e^{t-x}\left(1-\frac{1}{2}\erfc\left(\frac{x}{2\sqrt{t}}-\sqrt{t}\right)\right) + e^{t+x}\frac{1}{2}\erfc\left(\frac{x}{2\sqrt{t}}+\sqrt{t}\right),
\end{equation}
where
\begin{equation}
\label{eq:error_func}
    \erfc(x) = \frac{2}{\sqrt{\pi}}\int^{+\infty}_{x}e^{-y^2}dy.
\end{equation}
\end{theorem}
\begin{proof}
First, let us plug the formula giving the initial condition into the solution defined in \eqref{eq:poisson} and \eqref{eq:poisson_integral}. We have
\begin{equation*}
\begin{split}
    u(t,x) & = \frac{1}{2\sqrt{\pi t}}\int^{+\infty}_{-\infty}e^{-\frac{(x-y)^2}{4t}}u_0(y)dy \\
    & = \frac{1}{2\sqrt{\pi t}}\int^{+\infty}_{-\infty}e^{-\frac{(x-y)^2}{4t}}e^{-|y|}dy.
\end{split}
\end{equation*}
Now, change the variables within the integral by denoting $x - y = 2z\sqrt{t}$. It then yields that $y = x - 2z\sqrt{t}$ and $dy = 2\sqrt{t}dz$. So,
\begin{equation*}
\begin{split}
    u(t,x) & = \frac{1}{2\sqrt{\pi t}}\int^{+\infty}_{-\infty}e^{-z^2}e^{-|x-2z\sqrt{t}|}2\sqrt{t}dz \\
    & = \frac{1}{\sqrt{\pi}}\int^{\frac{x}{2\sqrt{t}}}_{-\infty}e^{-(z^2-2z\sqrt{t}+x)}dz + \frac{1}{\sqrt{\pi}}\int^{+\infty}_{\frac{x}{2\sqrt{t}}}e^{-(z^2+2z\sqrt{t}-x)}dz \\
    & = \frac{1}{\sqrt{\pi}}\int^{\frac{x}{2\sqrt{t}}}_{-\infty}e^{-(z-\sqrt{t})^2 + t - x}dz + \frac{1}{\sqrt{\pi}}\int^{+\infty}_{\frac{x}{2\sqrt{t}}}e^{-(z+\sqrt{t})^2 + t + x}dz.
\end{split}
\end{equation*}
Now, set $z - t = w$ and $z + t = w$ for the former respectively latter integral. Hence, we obtain
\begin{equation*}
\begin{split}
    u(t,x) & = \frac{1}{\sqrt{\pi}}e^{t-x}\int^{\frac{x}{2\sqrt{t}}-\sqrt{t}}_{-\infty}e^{-w^2}dw + \frac{1}{\sqrt{\pi}}e^{t+x}\int^{+\infty}_{\frac{x}{2\sqrt{t}} + \sqrt{t}}e^{-w^2}dw \\
    & = e^{t-x}\left(1-\frac{1}{2}\frac{2}{\sqrt{\pi}}\int^{+\infty}_{\frac{x}{2\sqrt{t}}-\sqrt{t}}e^{-w^2}dw\right) + \frac{1}{2}\frac{2}{\sqrt{\pi}}e^{t+x}\int^{+\infty}_{\frac{x}{2\sqrt{t}} + \sqrt{t}}e^{-w^2}dw \\
    & = e^{t-x}\left(1-\frac{1}{2}\erfc\left(\frac{x}{2\sqrt{t}}-\sqrt{t}\right)\right) + e^{t+x}\frac{1}{2}\erfc\left(\frac{x}{2\sqrt{t}}+\sqrt{t}\right).
\end{split}
\end{equation*}
Besides, as $\erfc(x)\xrightarrow{x\rightarrow+\infty}0$, then for each $t>0$ one has $u(t,x) \sim e^{t-x}$ as $x\to+\infty$ and, moreover, $u(t,x) \sim e^{t-|x|}$ as $x\to\infty$.
\end{proof}

As analysing the convergence speed for the obtained solution is technically a sophisticated task, we will present only the numerical estimates of the norm decay later on. Even though it is not a rigorous proof, it gives hope that one can obtain a formal reasoning.

\section{Numerical experiments on model examples}
This chapter is devoted to the results of numerical experiments performed for the previously analysed initial conditions $u_0 = [x\to \sin(x)]$ and $u_0 = [x \to \exp(-|x|)$, and different Cauchy problems \eqref{eq:cauchy_tr} and \eqref{def:cauchy_heat}.

\subsection{Transport equation}
\label{sec:tr_num}

\subsubsection{Convergence speed for \boldmath{$u_0 = [x \to \sin(x)]$}}
\label{sec:sin_tr_conv}
First, we will analyse the convergence speed for $u_0=[x\mapsto\sin(x)]$.

\begin{figure}[H]
    \centering
    \includegraphics[width=0.7\textwidth]{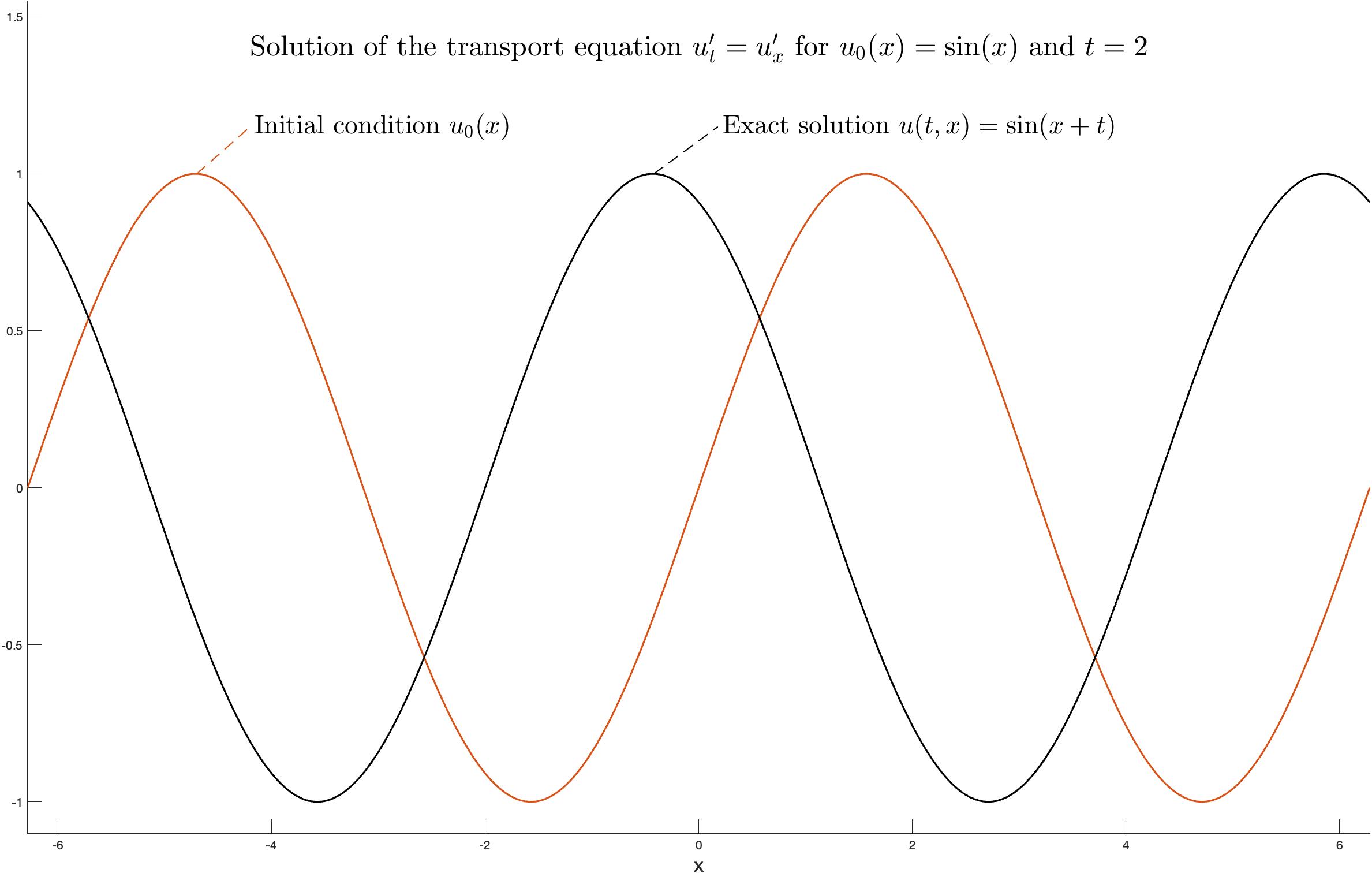}
    \caption{Graphs of initial condition and solution of the transport equation ($t = 2$)}
    \label{fig:tr_sin_solution}
\end{figure}

\begin{figure}[H]
    \centering
    \includegraphics[width=0.7\textwidth]{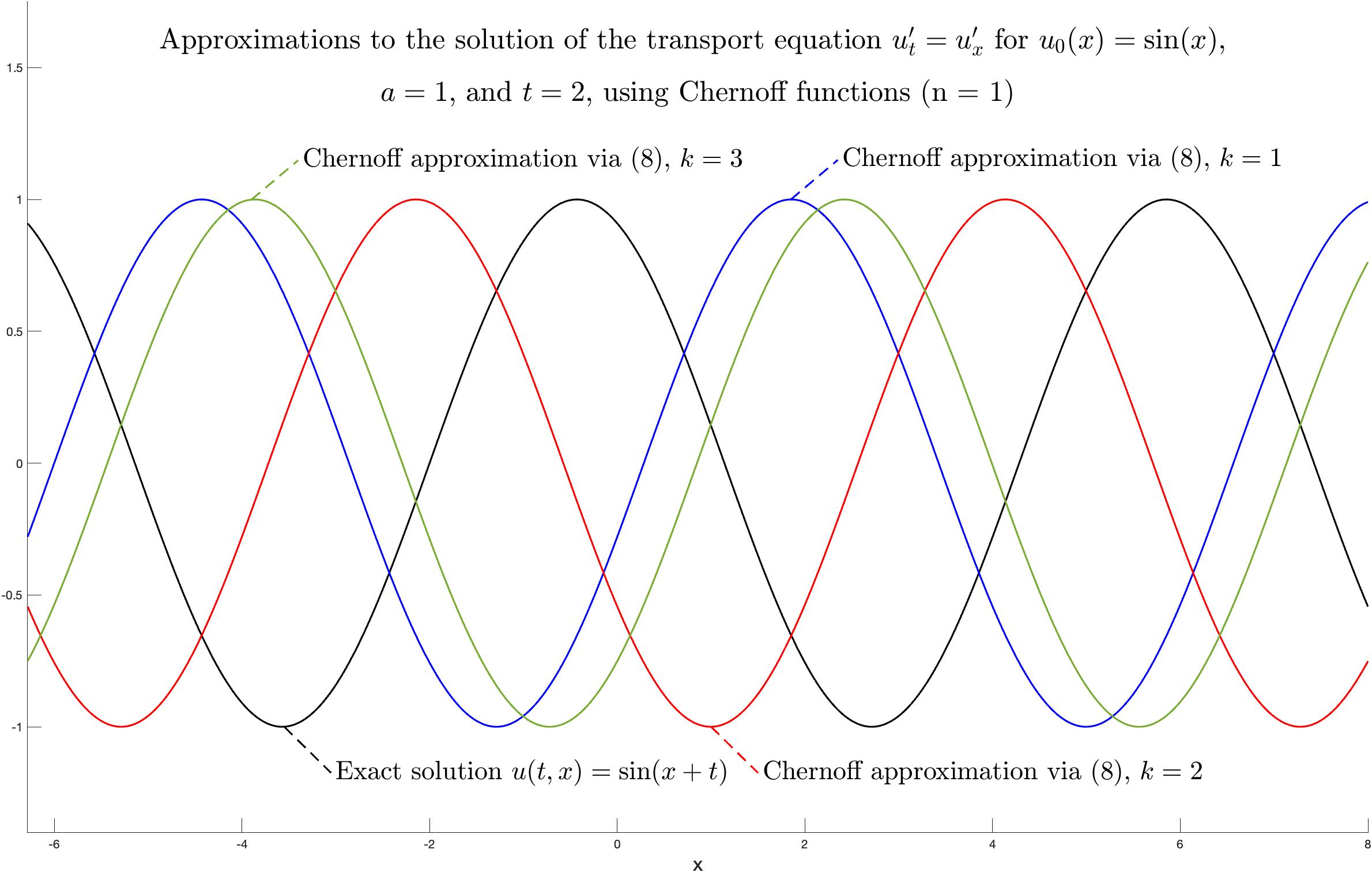}
    \caption{Graphs of approximations to the solution of the transport equation ($t = 2$, $n = 1$)}
    \label{fig:tr_sin_apprx_1}
\end{figure}

As the initial condition as well as the Chernoff approximations have period equal to $2\pi$, the uniform norm in $UC_b(\mathbb{R})$ is reached on the interval of this period. Note that the domain of the graphs presented below varies significantly and is chosen in favor of visual clarity.

So, Figure \ref{fig:tr_sin_solution} containes the graphs of the initial condition and the corresponding solution for $t = 2$. Whereas, Figures \ref{fig:tr_sin_apprx_1} and \ref{fig:tr_sin_apprx_5} represent some examples of Chernoff approximations with respect to the initial condition $u_0=[x\mapsto\sin(x)]$ and $t=2$ with composition degree $n = 1$ and $n = 5$ respectively.

\begin{figure}[H]
    \centering
    \includegraphics[width=0.75\textwidth]{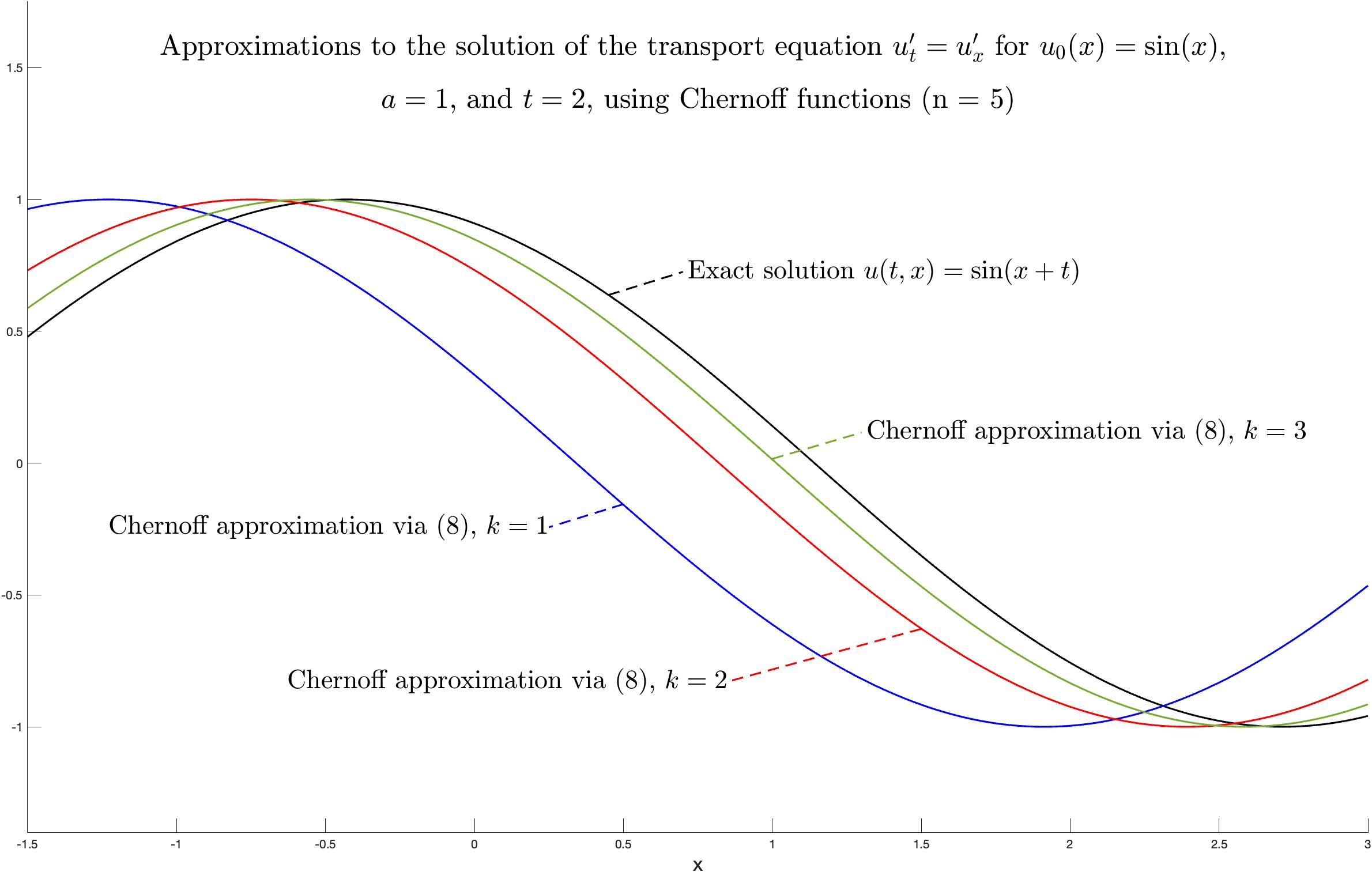}
    \caption{Graphs of approximations to the solution of the transport equation ($t = 2$, $n = 5$)}
    \label{fig:tr_sin_apprx_5}
\end{figure}

\begin{figure}[H]
    \centering
    \includegraphics[width=0.75\textwidth]{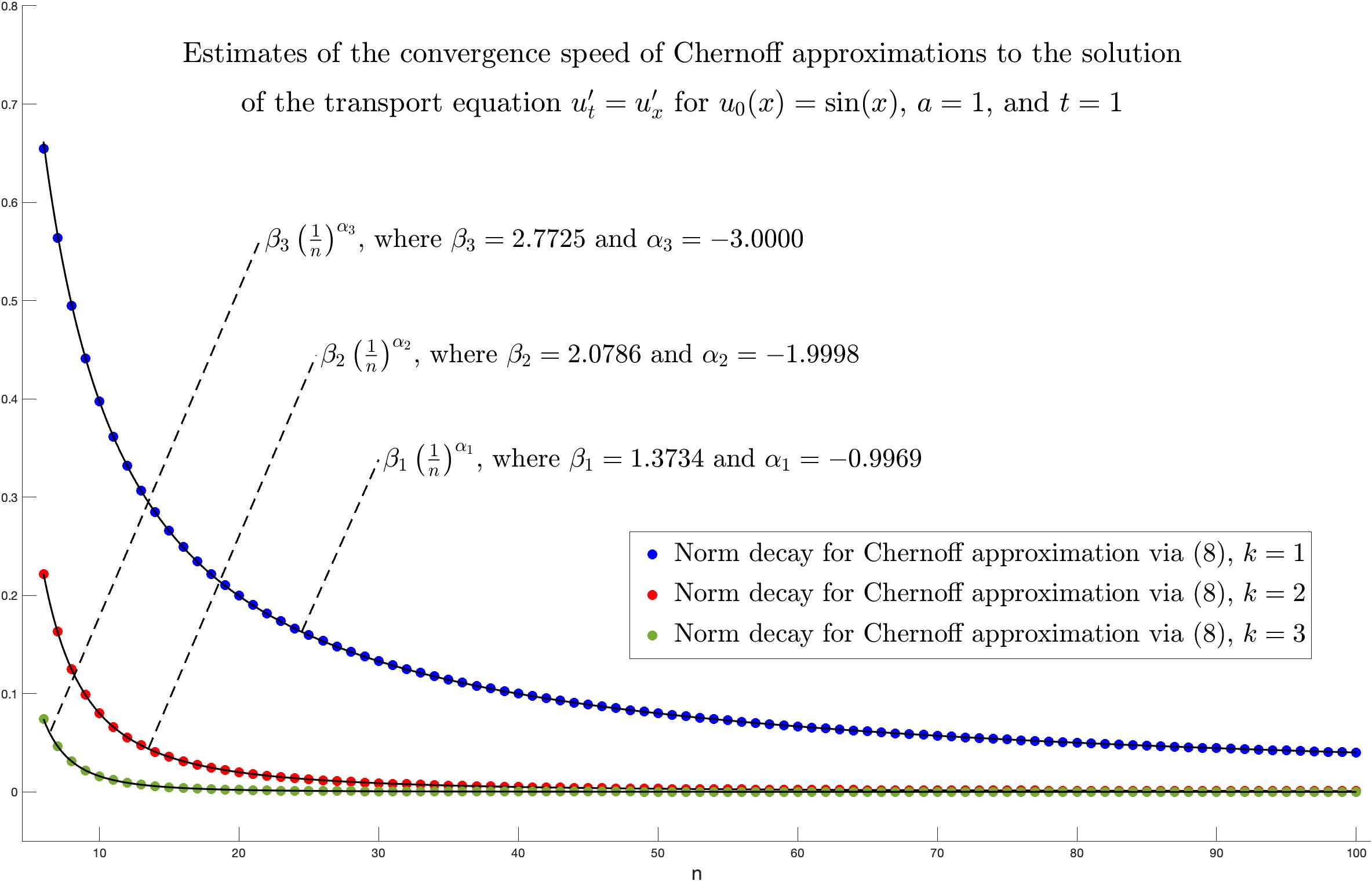}
    \caption{Graphs of estimates of the convergence speed ($t = 1$)}
    \label{fig:tr_sin_apprx_gen}
\end{figure}

Figure \ref{fig:tr_sin_apprx_gen} represents the graphs of convergence rate, i.e., the decay in the norm of the difference of the Chernoff approximations and the corresponding solution depending on the growth of composition degree.

\begin{figure}[H]
    \centering
    \includegraphics[width=0.71\textwidth]{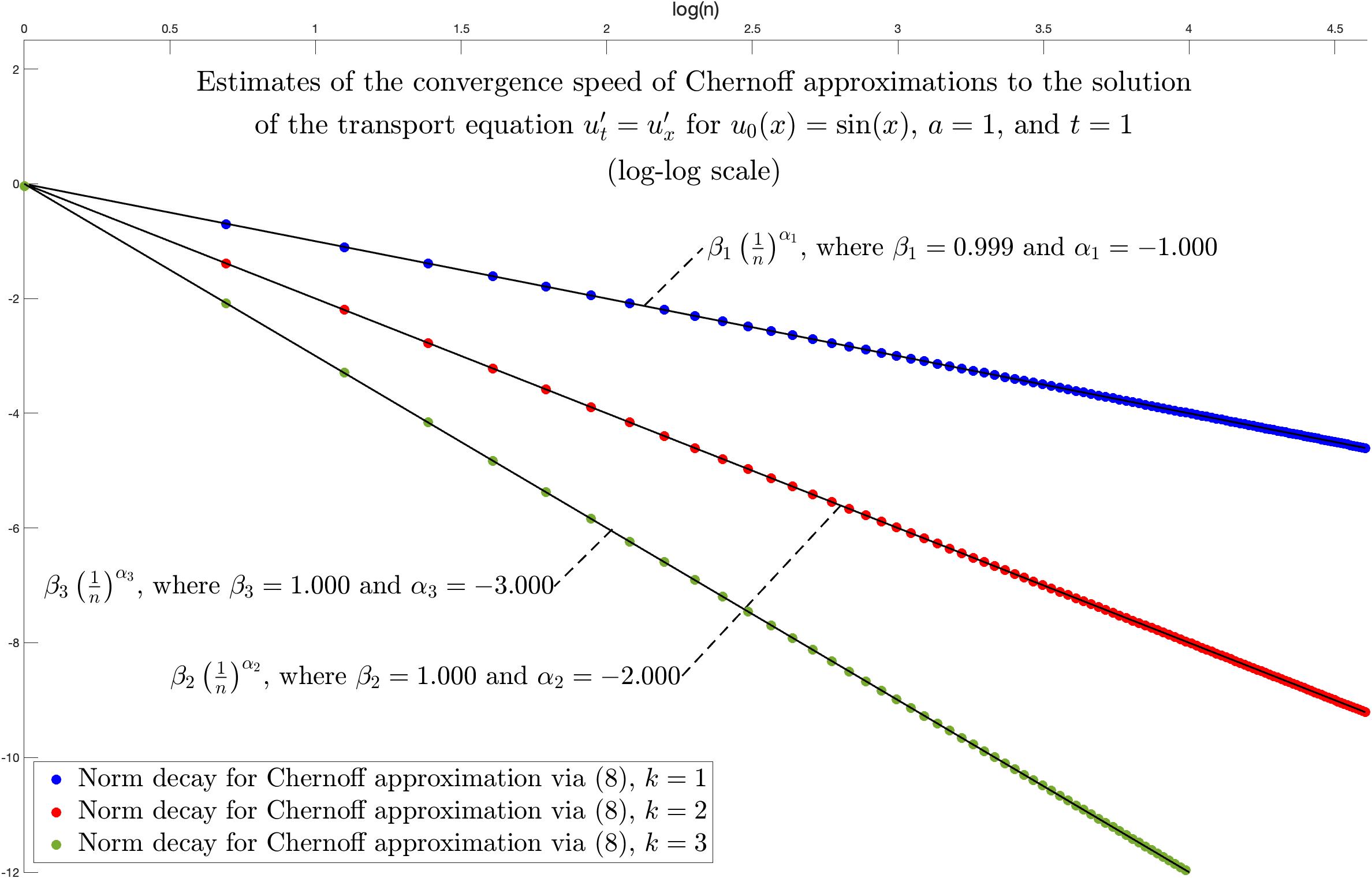}
    \caption{Graphs of estimates of the convergence speed in log-log scale ($t = 1$)}
    \label{fig:tr_sin_apprx_log}
\end{figure}

Apart from that, the above series of figures illustrates the results of a non-linear regression applied to analyse the order of approximation subspaces to which the examined initial condition belongs to. So, Figures \ref{fig:tr_sin_apprx_gen} and \ref{fig:tr_sin_apprx_log} using both standard and log-log scale represent the convergence rate and the corresponding obtained regression results for Chernoff functions from Theorem \ref{th:sin_tr}. 

\subsubsection{Convergence speed for \boldmath{$u_0 = [x \to \exp(-|x|)]$}}

\begin{figure}[H]
    \centering
    \includegraphics[width=0.71\textwidth]{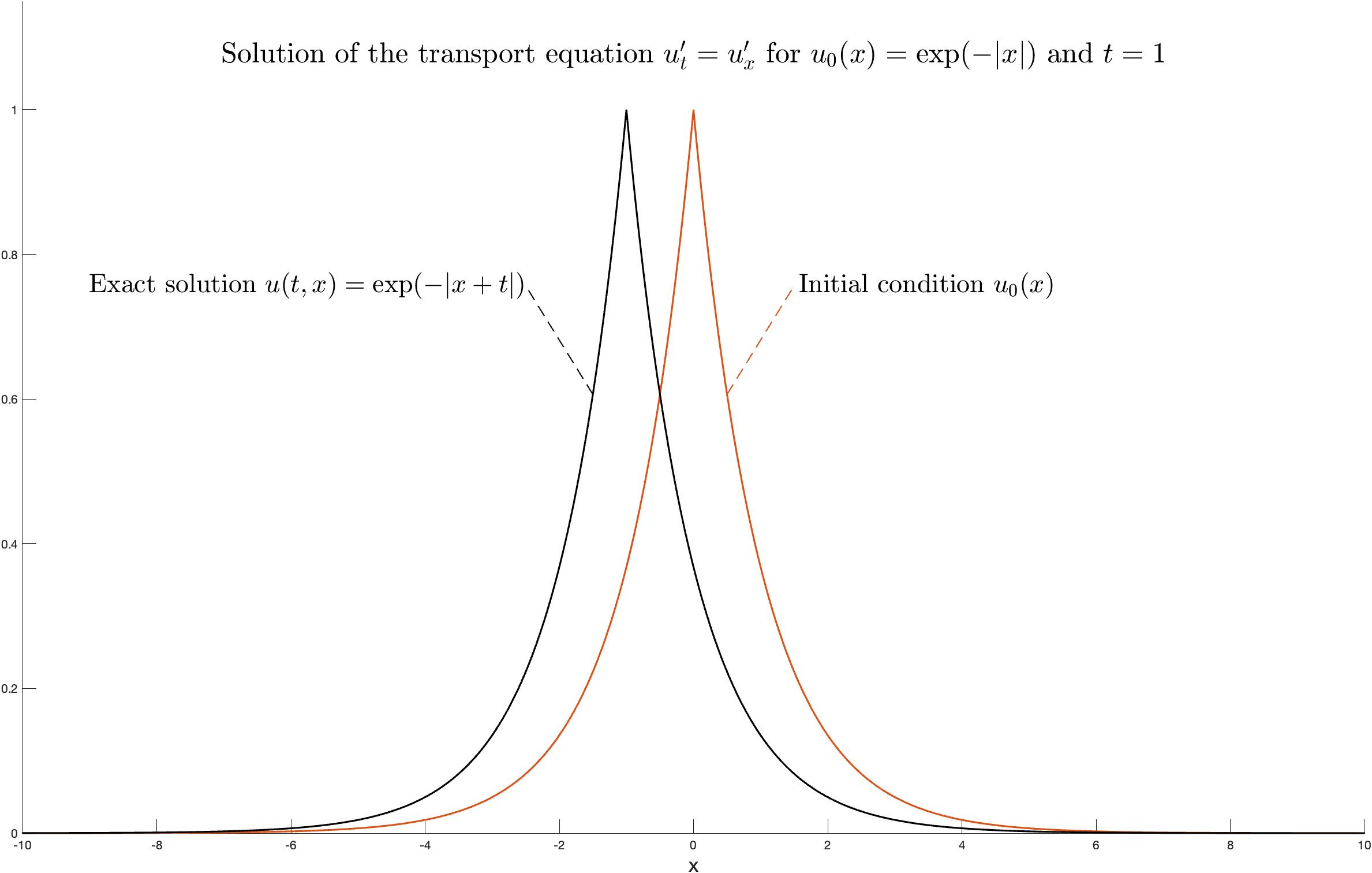}
    \caption{Graphs of initial condition and solution of the transport equation ($t = 1$)}
    \label{fig:tr_exp_solution}
\end{figure}

Second, we will analyse the convergence rate for $u_0=[x\mapsto e^{-|x|}]$. The graphs of the initial condition and examined solution for $t = 1$ are shown in Figure \ref{fig:tr_exp_solution}. 

\begin{figure}[H]
    \centering
    \includegraphics[width=0.8\textwidth]{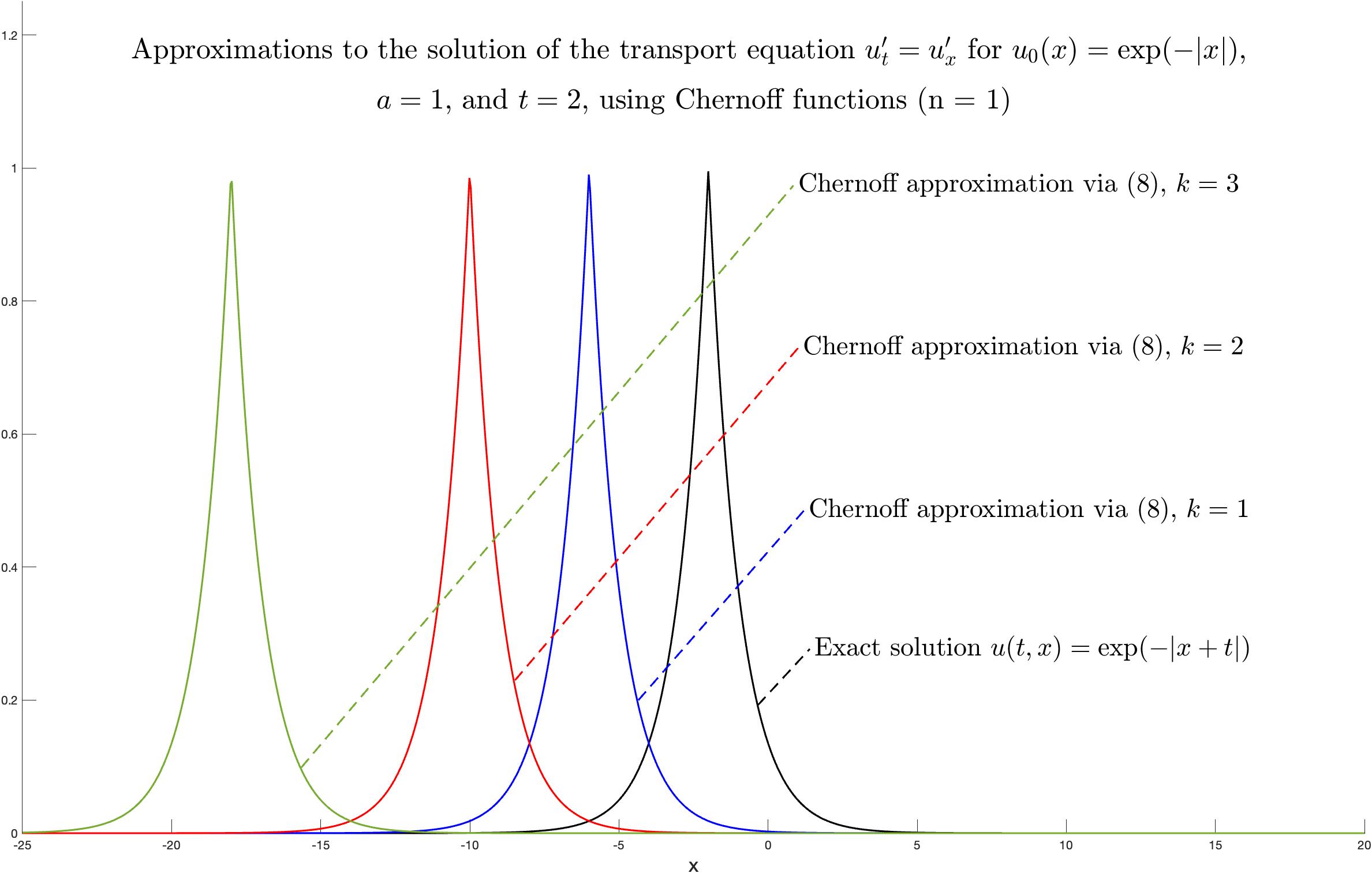}
    \caption{Graphs of approximations to the solution of the transport equation ($t = 2$, $n = 1$)}
    \label{fig:tr_exp_apprx_1}
\end{figure}

\begin{figure}[H]
    \centering
    \includegraphics[width=0.75\textwidth]{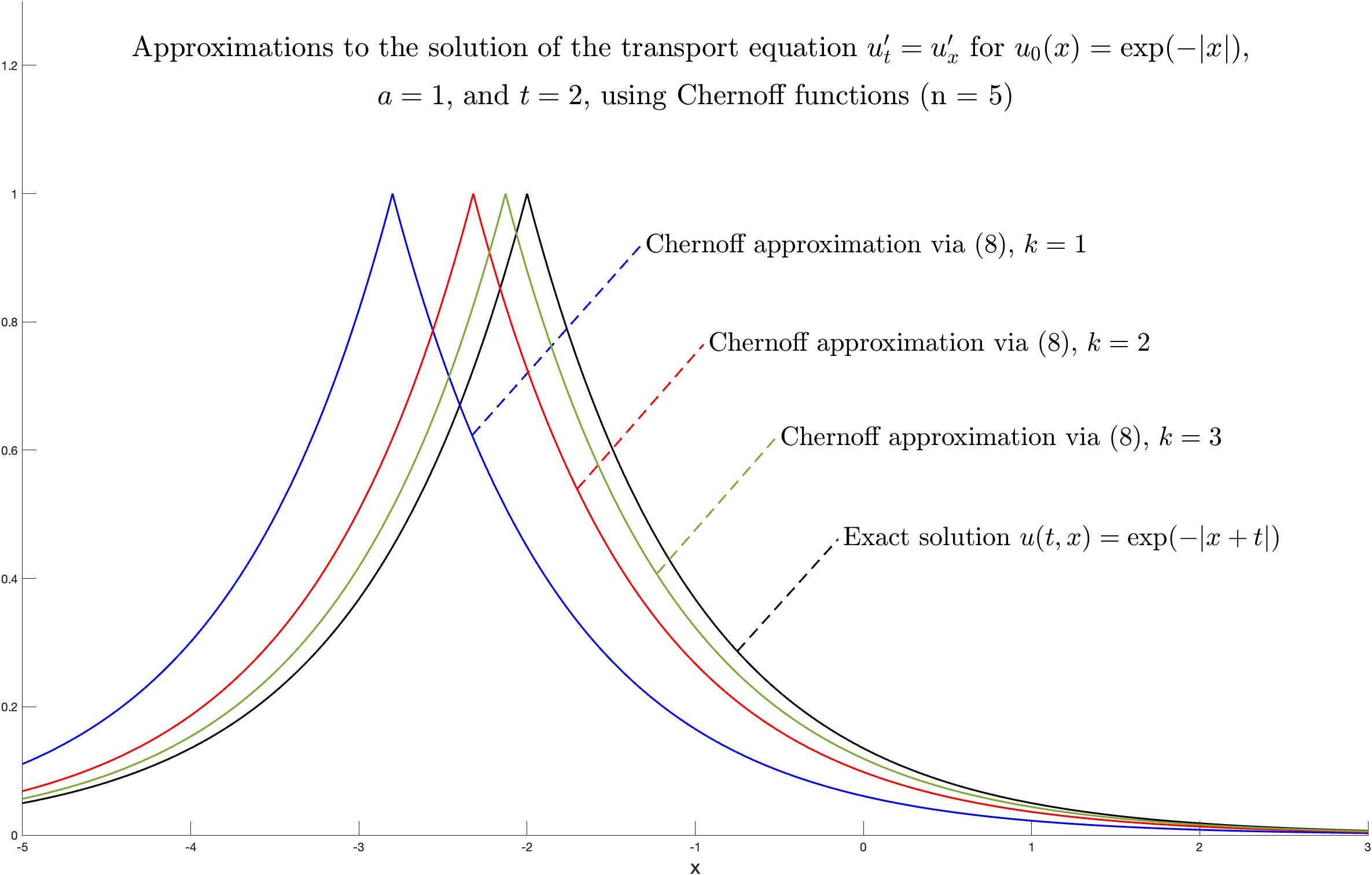}
    \caption{Graphs of approximations to the solution of the transport equation ($t = 2$, $n = 5$)}
    \label{fig:tr_exp_apprx_5}
\end{figure}

Figures \ref{fig:tr_exp_apprx_1} and  \ref{fig:tr_exp_apprx_5} represent some model examples of Chernoff approximations based on the initial condition $u_0=[x\mapsto e^{-|x|}]$ for $t = 2$ with composition degree $n = 1$ and $n = 5$ respectively.

As the initial condition and the Chernoff approximations are decaying outside of some segment, we assume that the value of the uniform norm of the difference is reached on the interval of correctly selected length (e.g., $[-5, 5]$). So, Figure \ref{fig:tr_exp_apprx_gen} provides a graph of convergence rate, i.e., the decay of the norm of the difference depending on the growth of composition degree.

\begin{figure}[H]
    \centering
    \includegraphics[width=0.75\textwidth]{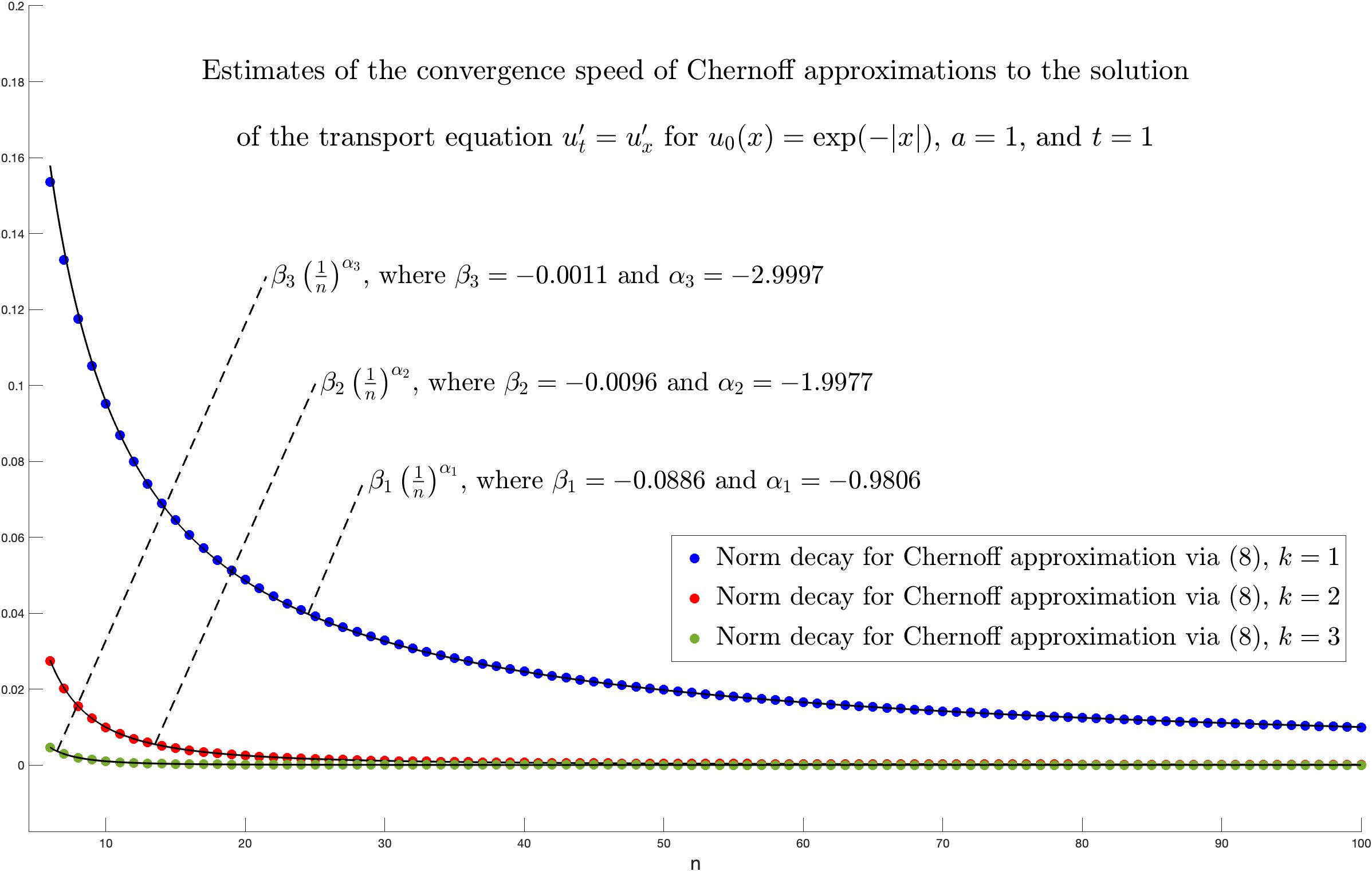}
    \caption{Graphs of estimates of the convergence speed ($t = 1$)}
    \label{fig:tr_exp_apprx_gen}
\end{figure}

\begin{figure}[H]
    \centering
    \includegraphics[width=0.75\textwidth]{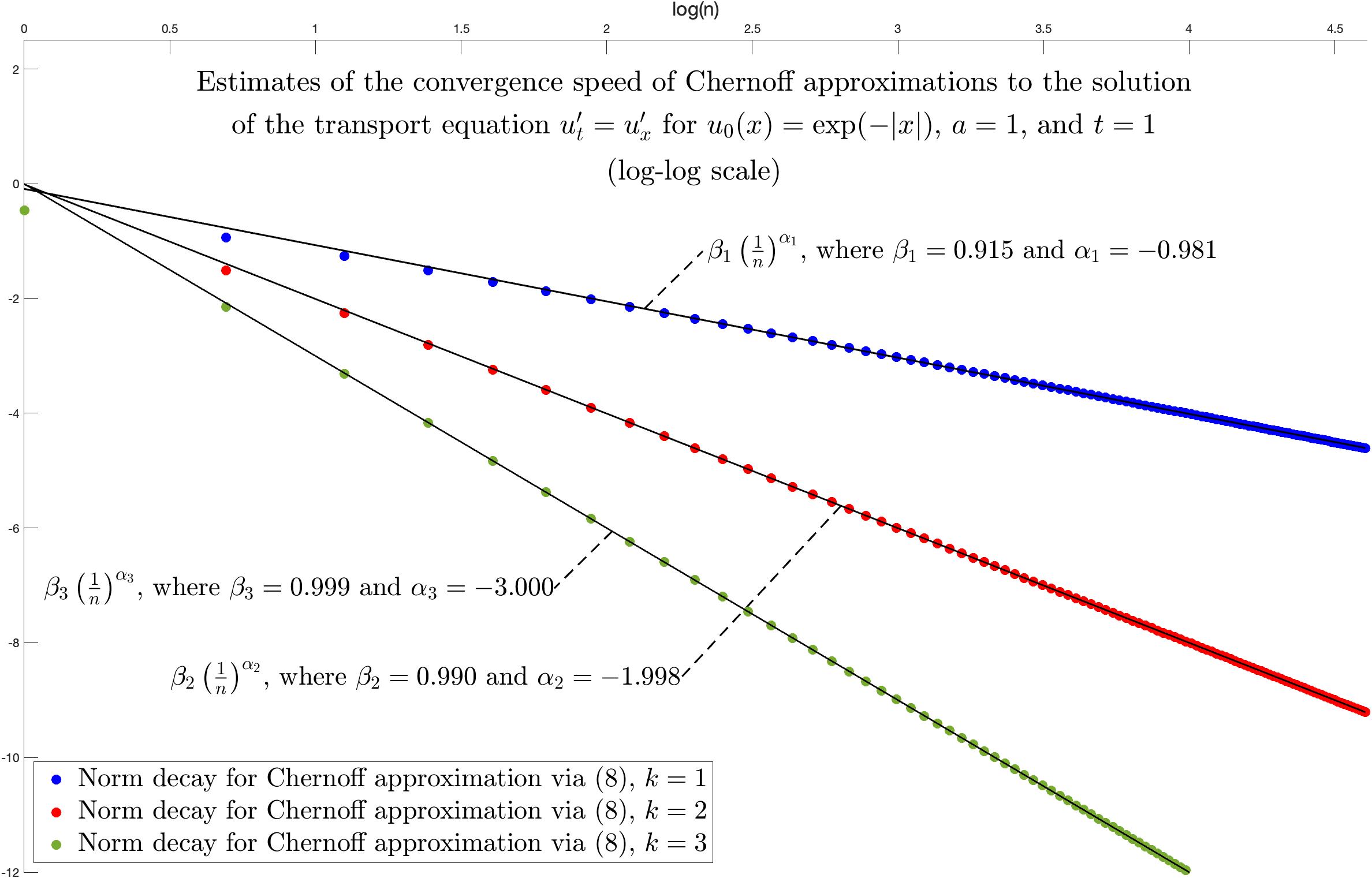}
    \caption{Graphs of estimates of the convergence speed in log-log scale ($t = 1$)}
    \label{fig:tr_exp_apprx_log}
\end{figure}

More than that, Figures \ref{fig:tr_exp_apprx_gen} and \ref{fig:tr_exp_apprx_log}, using both standard and log-log scale, illustrate the results of a non-linear regression applied to analyse the approximation subspaces to which the examined initial condition belongs to for Chernoff functions from Theorem \ref{th:sin_tr}. 

\subsubsection{Numerical demonstration of slow convergence for \boldmath{$u_0 = [x\to\sin(x)]$}}
Third, we will analyse the convergence rate for $u_0=[x\mapsto \sin(x)]$ and Chernoff function $G$ \eqref{eq:tr_apprx_slow} given in Theorem \ref{th:sin_slow_tr} with $w(t) = \frac{1}{t^{\gamma}}$ for $\gamma = \frac{1}{2}$, $\gamma = \frac{1}{3}$, and $\gamma = \frac{1}{6}$.

\begin{figure}[H]
    \centering
    \captionsetup{justification=centering}
    \includegraphics[width=0.75\textwidth]{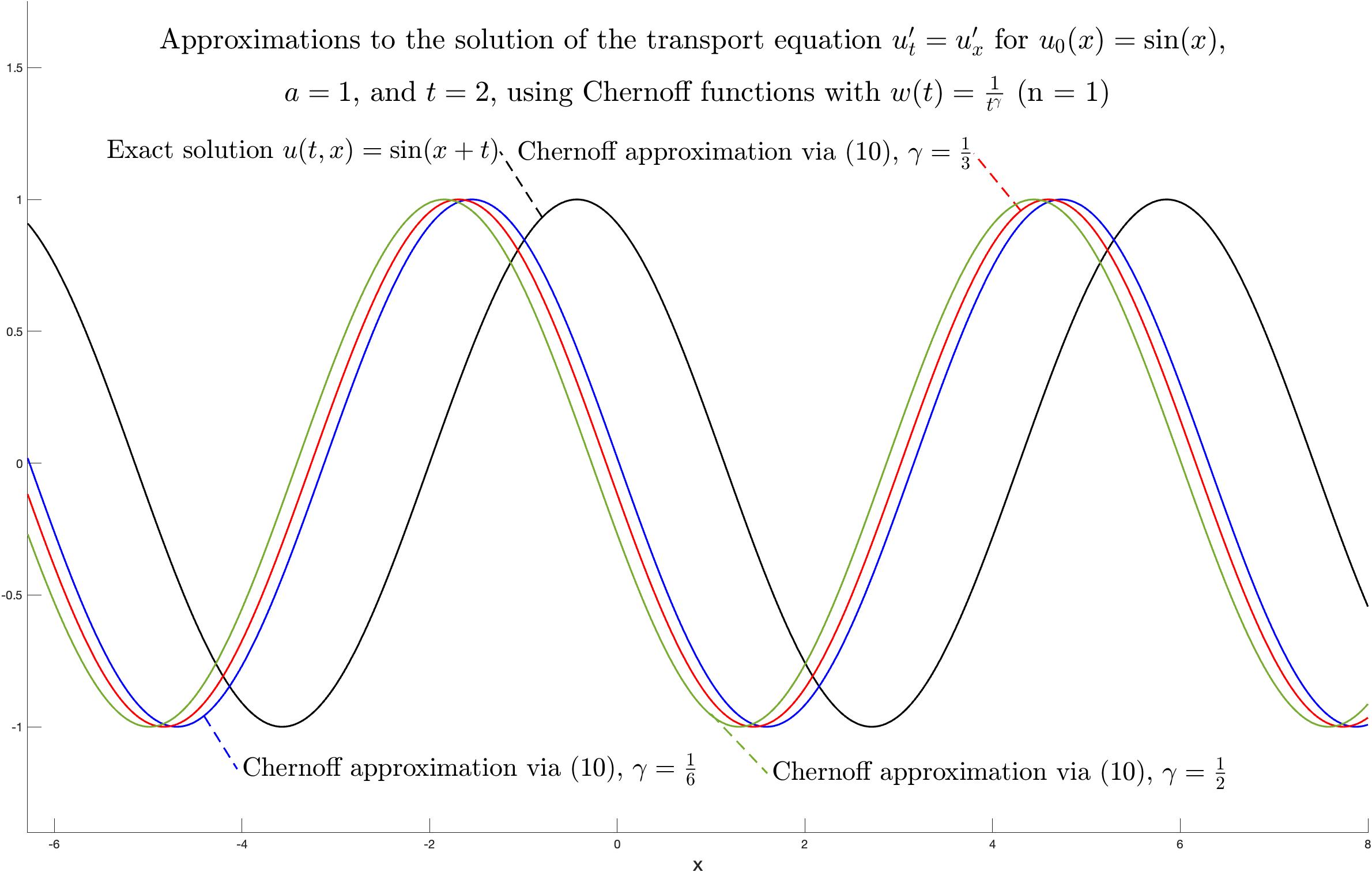}
    \caption{Graphs of approximations to the solution of the transport equation ($t = 2$, $n = 1$)}
    \label{fig:tr_sin_slow_apprx_1}
\end{figure}

\begin{figure}[H]
    \centering
    \captionsetup{justification=centering}
    \includegraphics[width=0.75\textwidth]{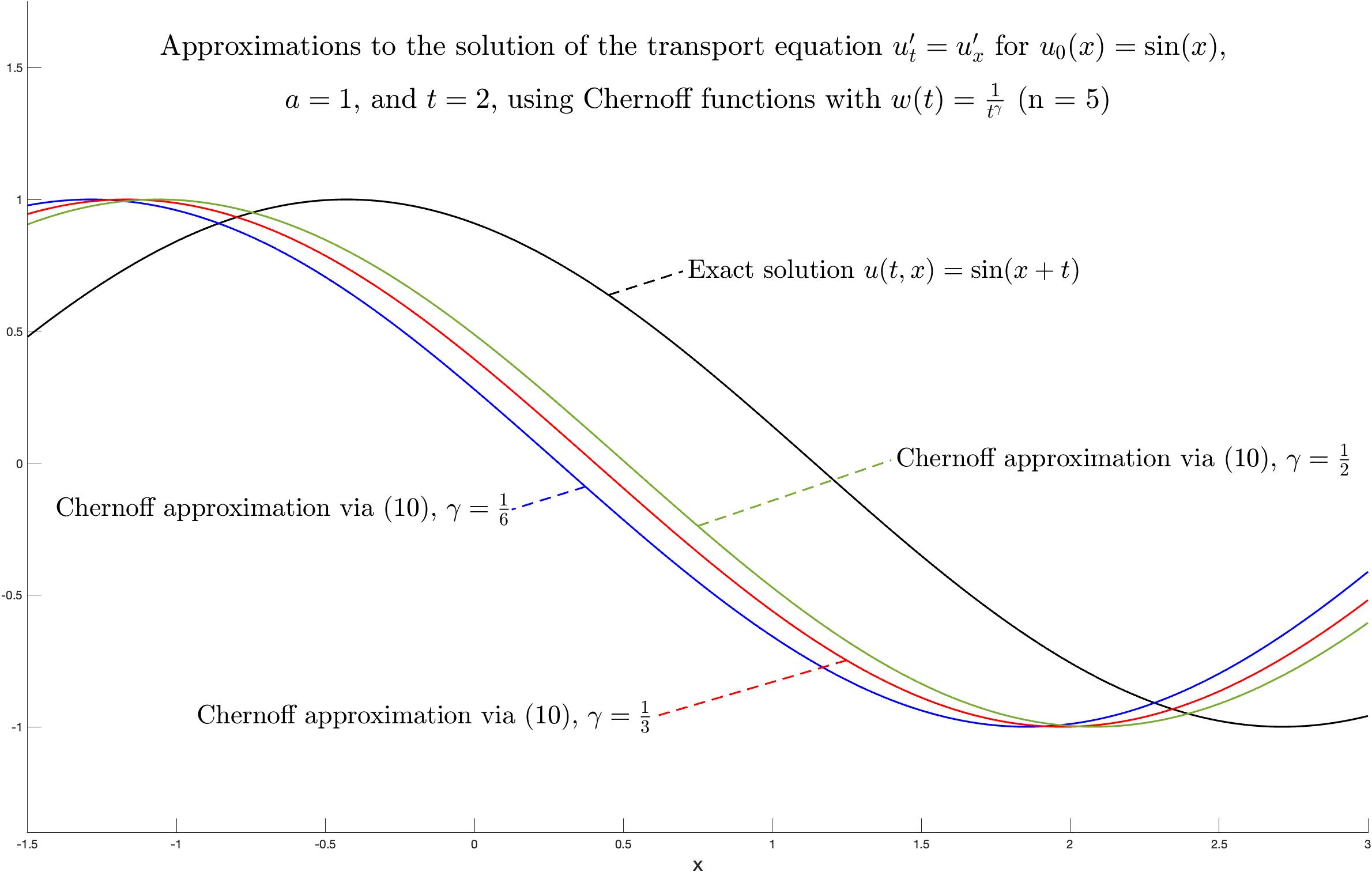}
    \caption{Graphs of approximations to the solution of the transport equation ($t = 2$, $n = 5$)}
    \label{fig:tr_sin_slow_apprx_5}
\end{figure}

Figures \ref{fig:tr_sin_slow_apprx_1} and \ref{fig:tr_sin_slow_apprx_5} demonstrate some common examples of Chernoff approximations for the initial condition $u_0=[x\mapsto \sin(x)]$ and $t=2$ with composition degree $n = 1$ and $n = 5$ respectively.

Here, we will again exploit the periodic nature of the solution and the Chernoff approximations and suggest that the uniform norm is reached on the properly selected interval. So, Figure \ref{fig:tr_sin_slow_apprx_gen} provides a graph of convergence rate when $t = 2$.

Apart from that, the below series of figures also illustrates the order of approximation subspaces to which the examined initial condition belongs to. Thus, Figures \ref{fig:tr_sin_slow_apprx_gen} and \ref{fig:tr_sin_slow_apprx_log}, using both standard and log-log scale, represent the convergence rate and the corresponding obtained regression results for Chernoff functions from Theorem \ref{th:sin_slow_tr} with $\gamma = \frac{1}{2}$, $\gamma = \frac{1}{3}$, and $\gamma = \frac{1}{6}$. 

\begin{figure}[H]
    \centering
    \captionsetup{justification=centering}
    \includegraphics[width=0.75\textwidth]{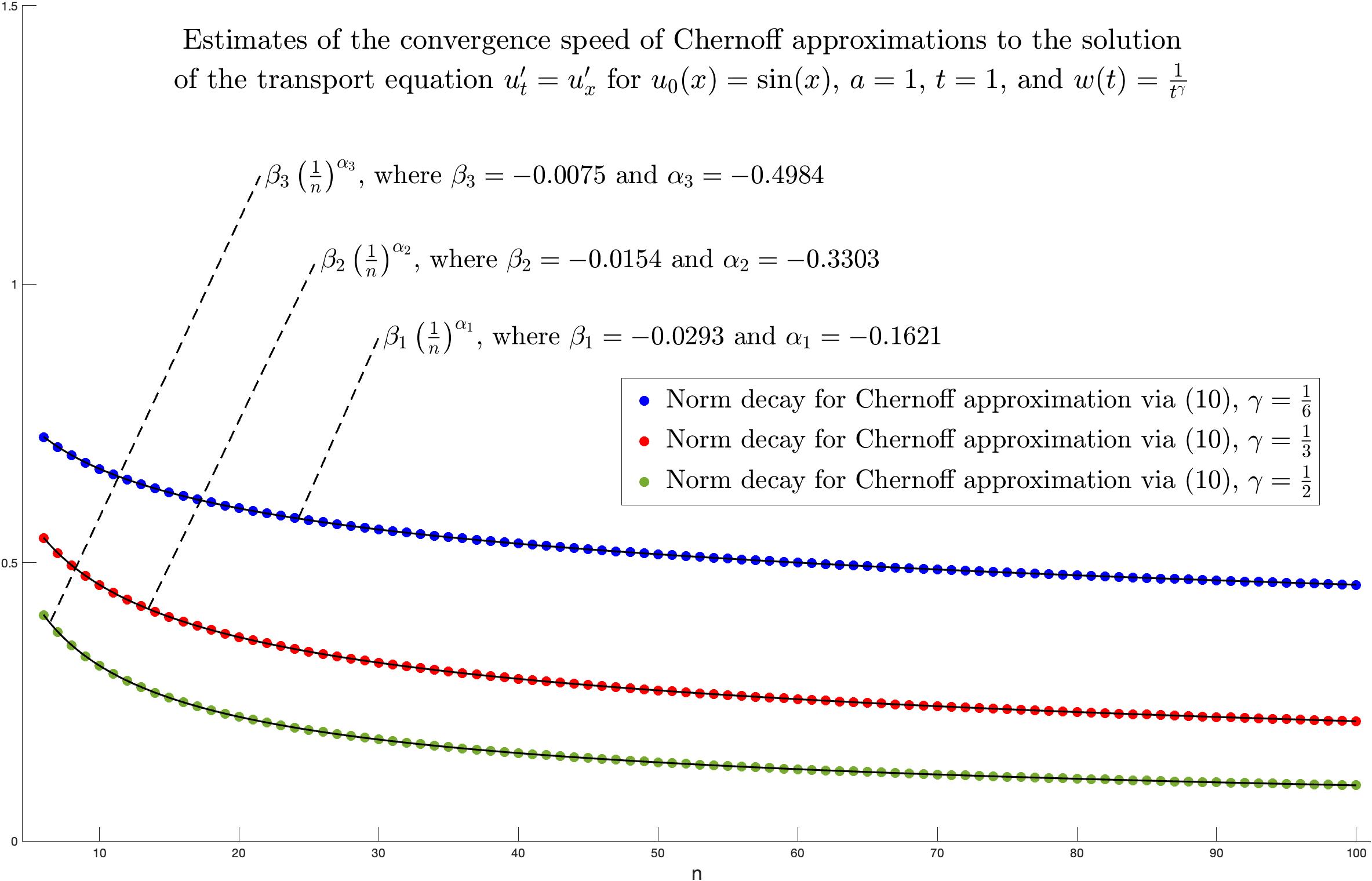}
    \caption{Graphs of estimates of the convergence speed ($t = 1$)}
    \label{fig:tr_sin_slow_apprx_gen}
\end{figure}

\begin{figure}[H]
    \centering
    \captionsetup{justification=centering}
    \includegraphics[width=0.75\textwidth]{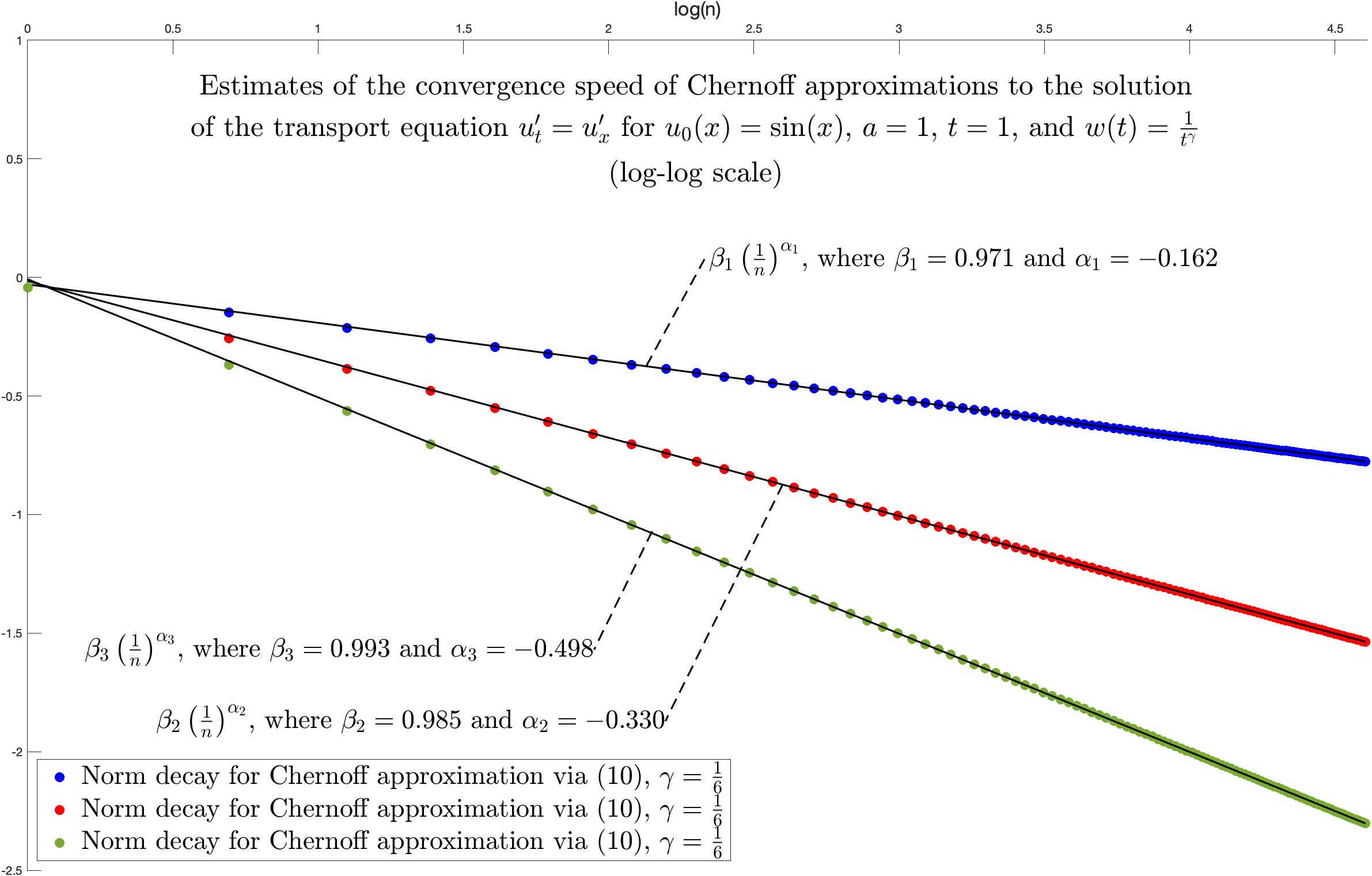}
    \caption{Graphs of estimates of the convergence speed with log-log scale ($t = 1$)}
    \label{fig:tr_sin_slow_apprx_log}
\end{figure}

One can clearly notice that, independent of the considered initial value of the equation, all of the approximations performed as if the examined model function belonged to the domain of generator, i.e., even $u_0 = [x\to\exp(-|x|)]\notin UC^{1}_{b}(\mathbb{R})$ belongs to the higher order approximation subspaces. We will examine later in this chapter that for the heat equation this is not the case already.

\subsection{Heat equation}
\label{sec:heat_num}

\subsubsection{Convergence speed for  \boldmath{$u_0 = [x \to \sin(x)]$}}

\begin{figure}[H]
    \centering
    \includegraphics[width=0.75\textwidth]{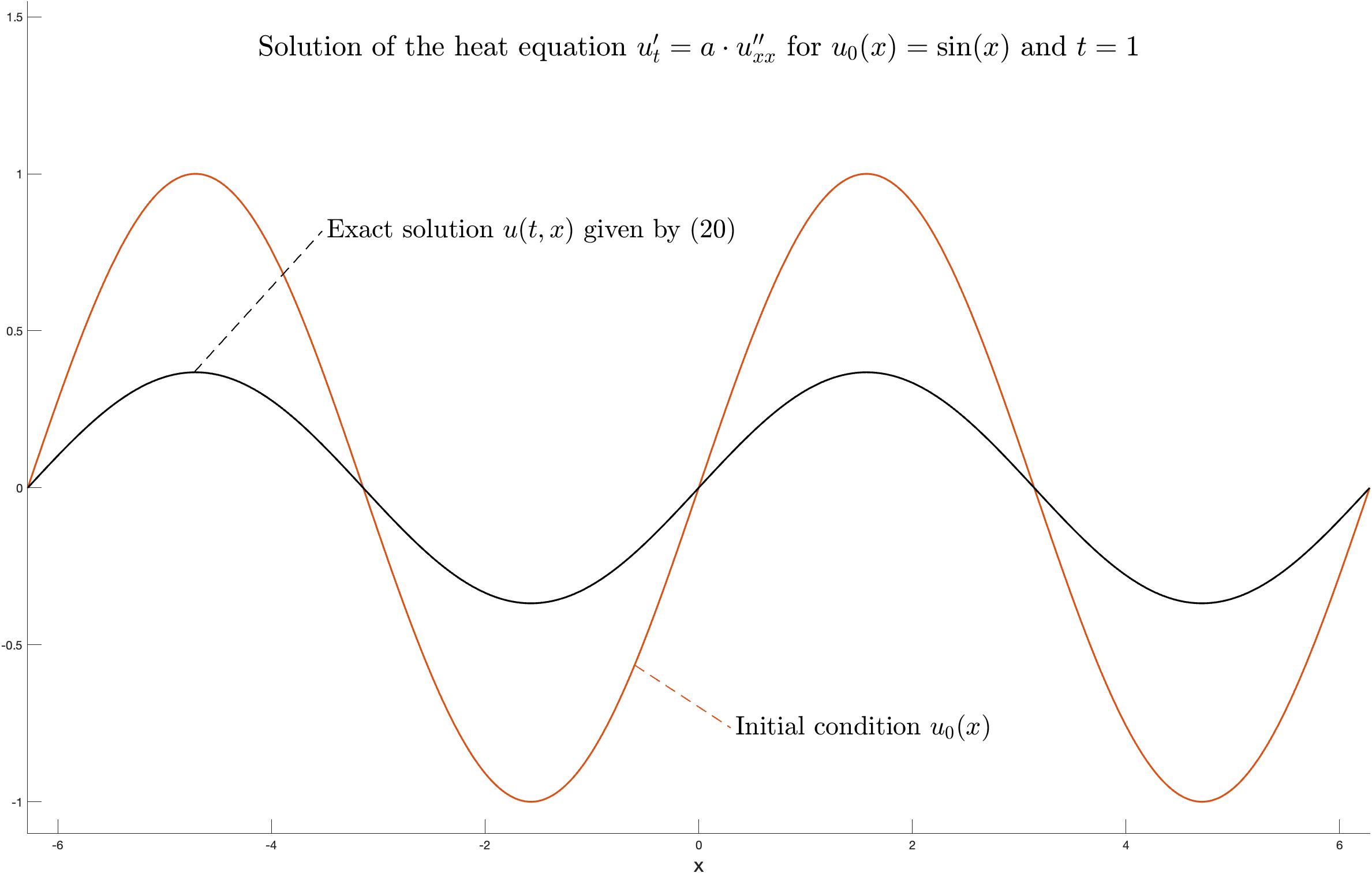}
    \caption{Graphs of initial condition and solution of the heat equation ($t = 1$)}
    \label{fig:heat_sin_solution}
\end{figure}

\begin{figure}[H]
    \centering
    \includegraphics[width=0.75\textwidth]{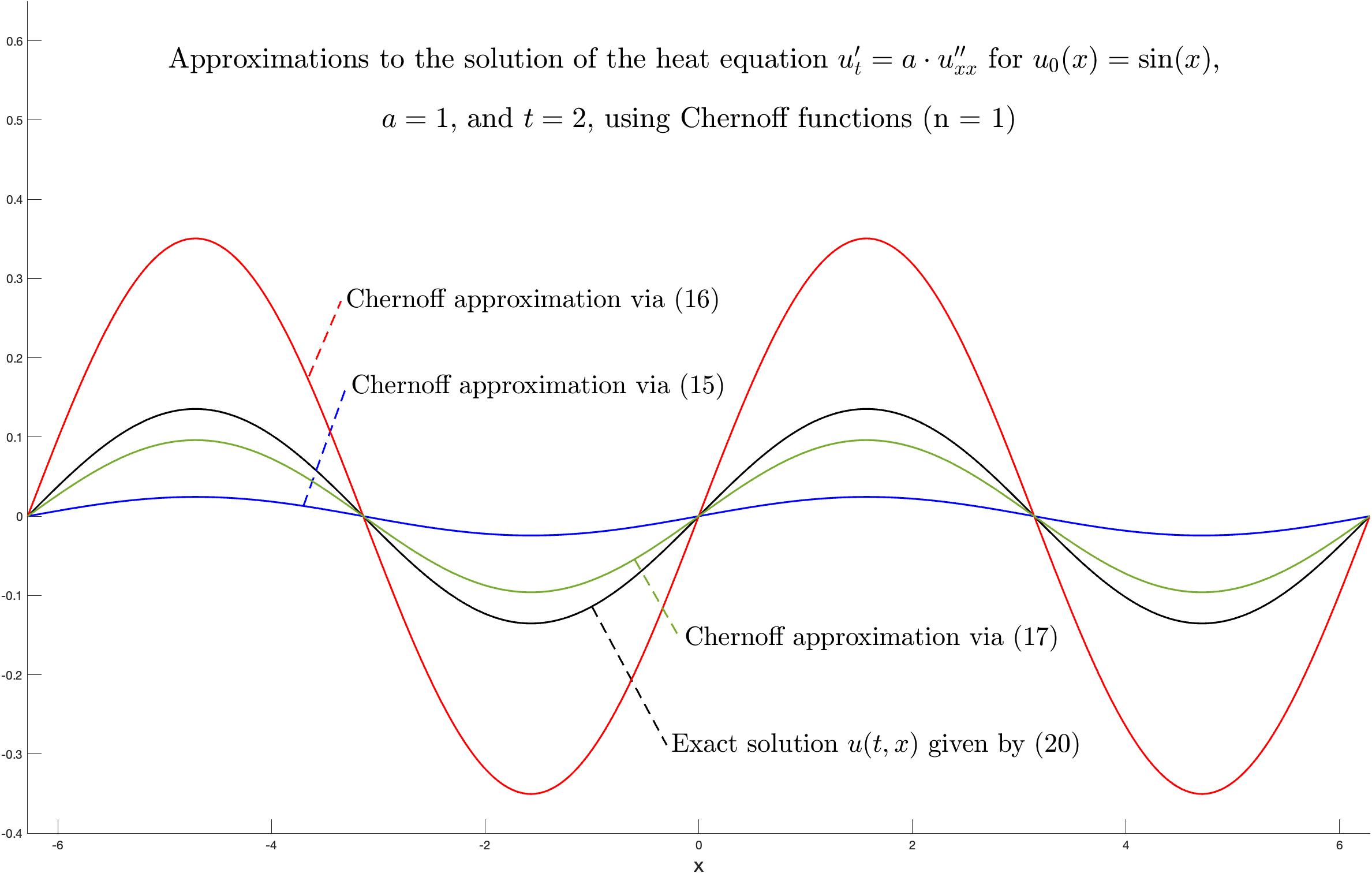}
    \caption{Graphs of approximations to the solution of the heat equation ($t = 2$, $n = 1$)}
    \label{fig:heat_sin_apprx_1}
\end{figure}

Here, we will analyse the convergence speed of Chernoff approximations defined by \eqref{eq:chernoff_heat_first}, \eqref{eq:chernoff_heat_second}, and \eqref{eq:chernoff_heat_third} for $u_0=[x\mapsto \sin(x)]$ to the exact solution given by \eqref{eq:heat_semigroup_sin}.

The graphs of the initial condition and investigated solution for $t=1$ are shown in Figure \ref{fig:heat_sin_solution}. 

\begin{figure}[H]
    \centering
    \includegraphics[width=0.75\textwidth]{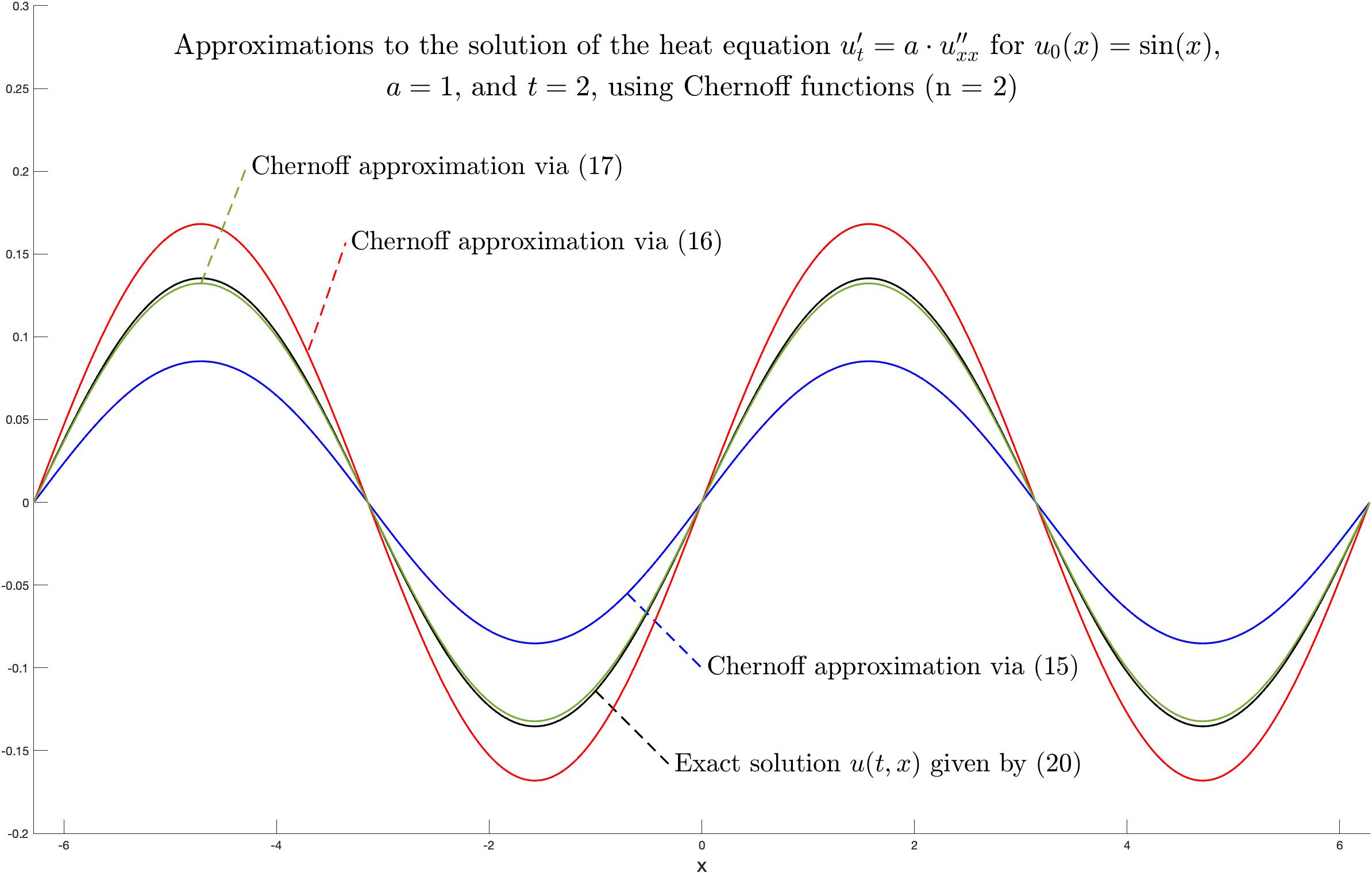}
    \caption{Graphs of approximations to the solution of the heat equation ($t = 2$, $n = 2$)}
    \label{fig:heat_sin_apprx_2}
\end{figure}

Whereas, Figures \ref{fig:heat_sin_apprx_1} and \ref{fig:heat_sin_apprx_2} represent some model examples of Chernoff approximations for the initial condition $u_0 = [x\mapsto \sin(x)]$ and $t=2$ with composition degree $n = 1$ and $n = 2$ respectively.

Exploiting one more time the periodic nature of the initial condition and the Chernoff approximations, we state that the uniform norm in $UC_b(\mathbb{R})$ is reached on the interval of the proper period, e.g., $[0, 2\pi]$. 

\begin{figure}[H]
    \centering
    \includegraphics[width=0.725\textwidth]{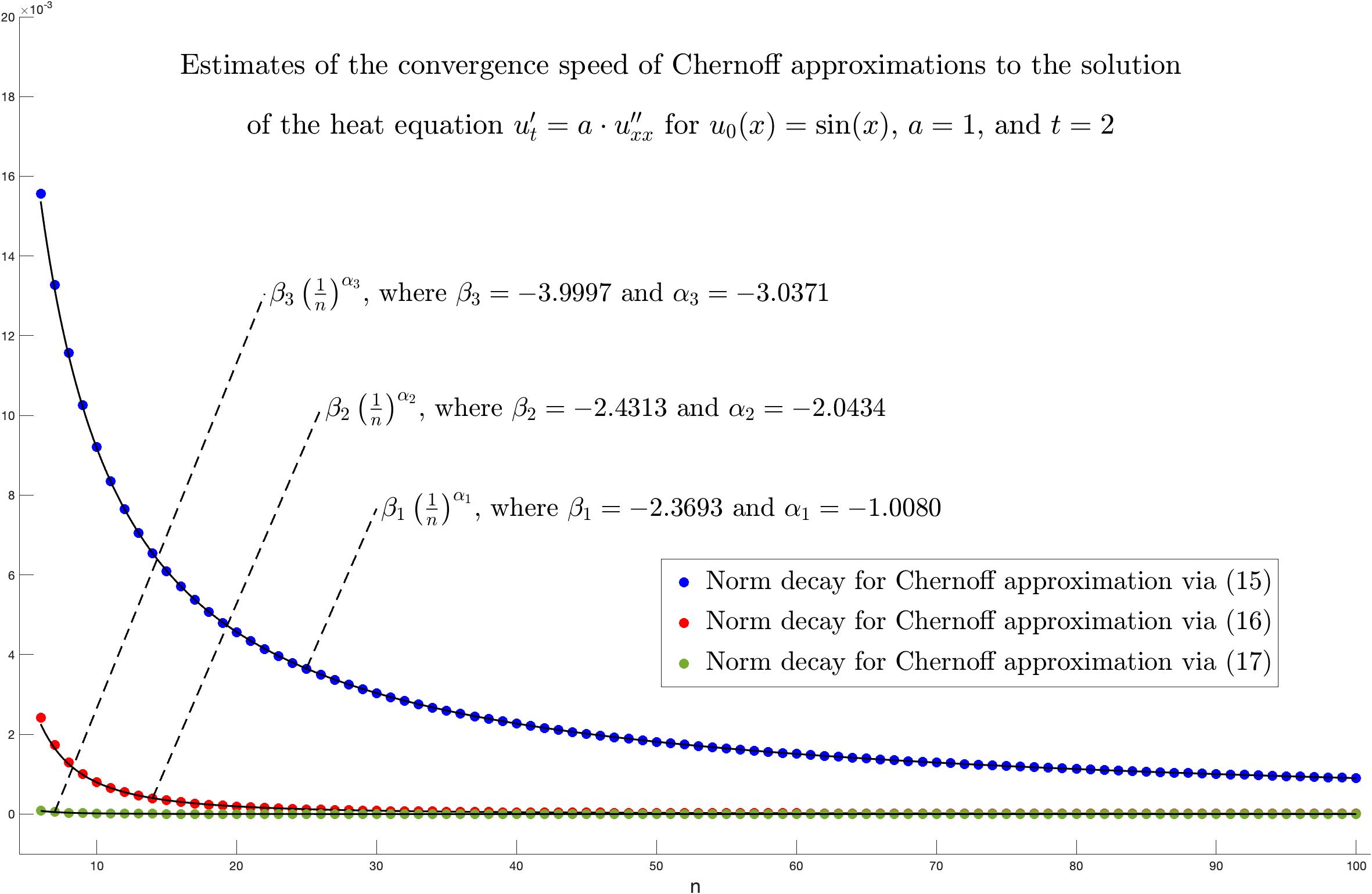}
    \caption{Graphs of estimates of the convergence speed ($t = 2$)}
    \label{fig:heat_sin_apprx_gen}
\end{figure}

\begin{figure}[H]
    \centering
    \includegraphics[width=0.75\textwidth]{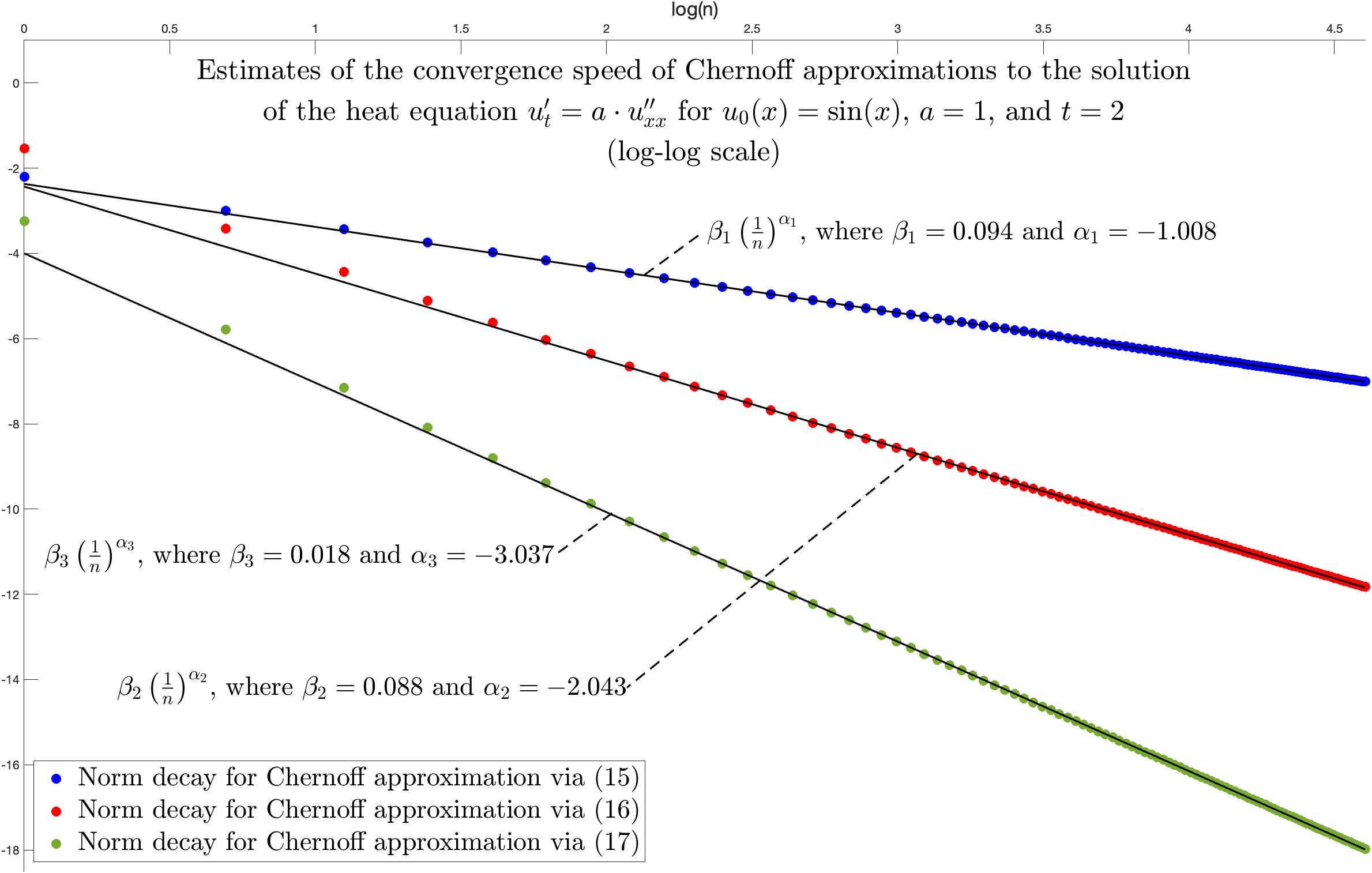}
    \caption{Graphs of estimates of the convergence speed with log-log scale ($t = 2$)}
    \label{fig:heat_sin_apprx_log}
\end{figure}

Figures \ref{fig:heat_sin_apprx_gen} and \ref{fig:heat_sin_apprx_log}, using both standard and logarithmic scale, represent the convergence rate and the corresponding obtained regression results for Chernoff functions from Theorem \ref{th:sin_first_heat} and Proposition \ref{th:sin_second_heat} for $t = 2$. Hence, one we can state that the results of numerical experiments confirmed the estimates obtained in Section 3.

\subsubsection{Convergence speed for \boldmath{$u_0 = [x \to \exp(-|x|)]$}}

\begin{figure}[H]
    \centering
    \includegraphics[width=0.75\textwidth]{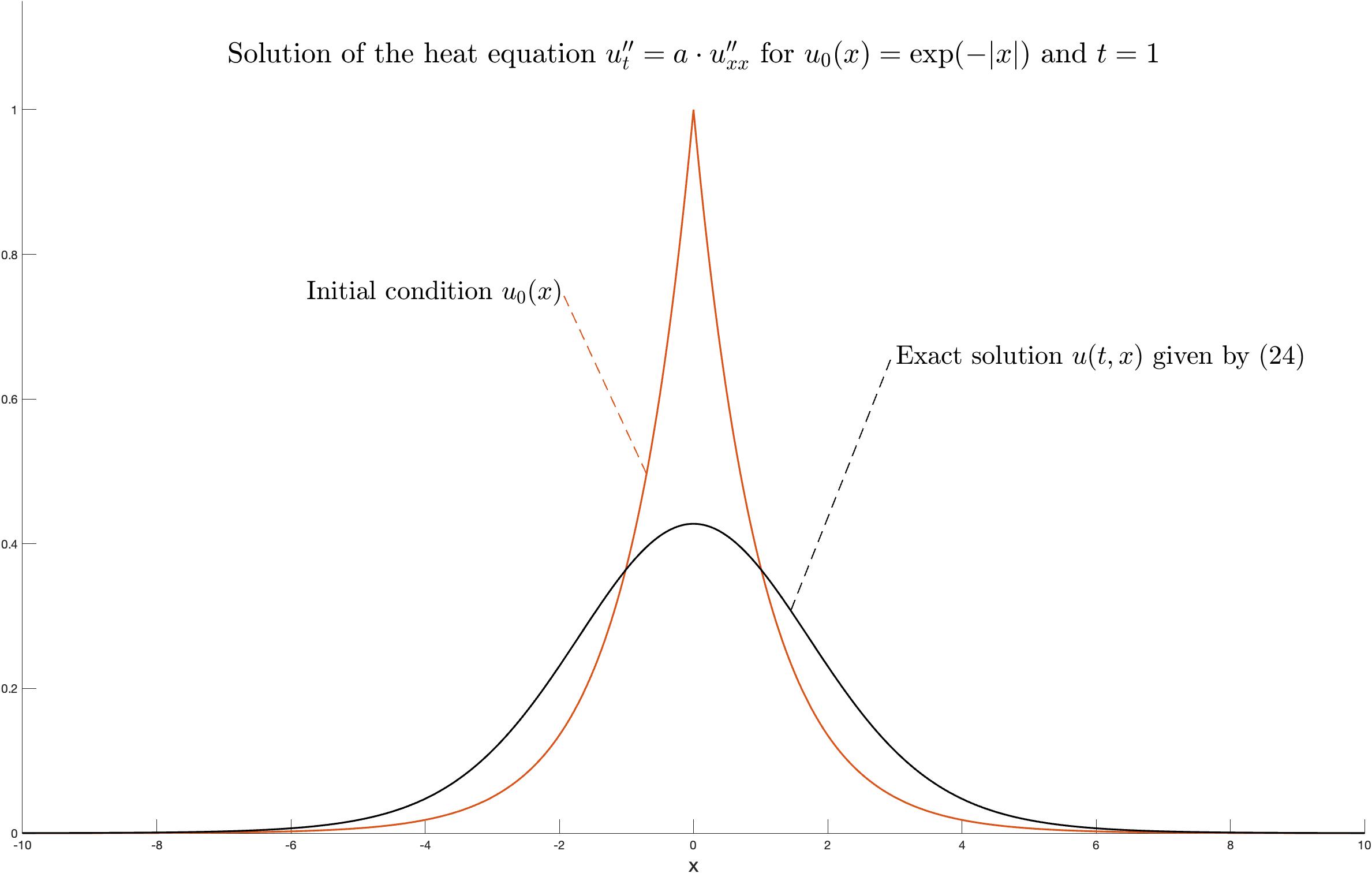}
    \caption{Graphs of initial condition and solution of the heat equation ($t = 1$)}
    \label{fig:heat_exp_solution}
\end{figure}

Here, we will analyse the convergence rate for the initial condition $u_0 = [x\mapsto e^{-|x|}]$.
The graphs of the initial condition and solution given by the $C_0$-semigroup for $t = 1$ are shown in Figure \ref{fig:heat_exp_solution}. 

\begin{figure}[H]
    \centering
    \includegraphics[width=0.75\textwidth]{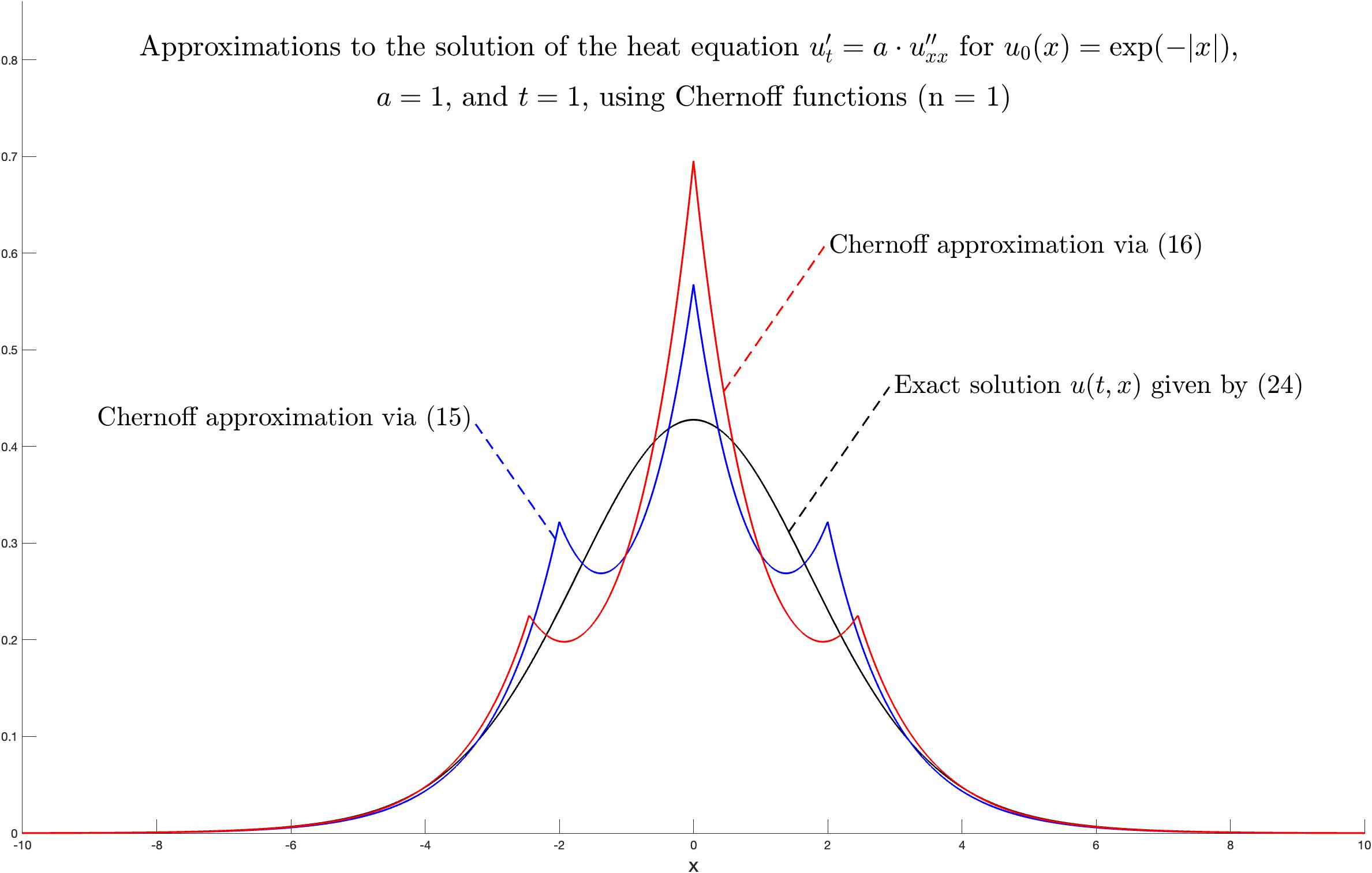}
    \caption{Graphs of approximations to the solution of the heat equation ($t = 1$, $n = 1$)}
    \label{fig:heat_exp_apprx_1}
\end{figure}
\begin{figure}[H]
    \centering
    \includegraphics[width=0.75\textwidth]{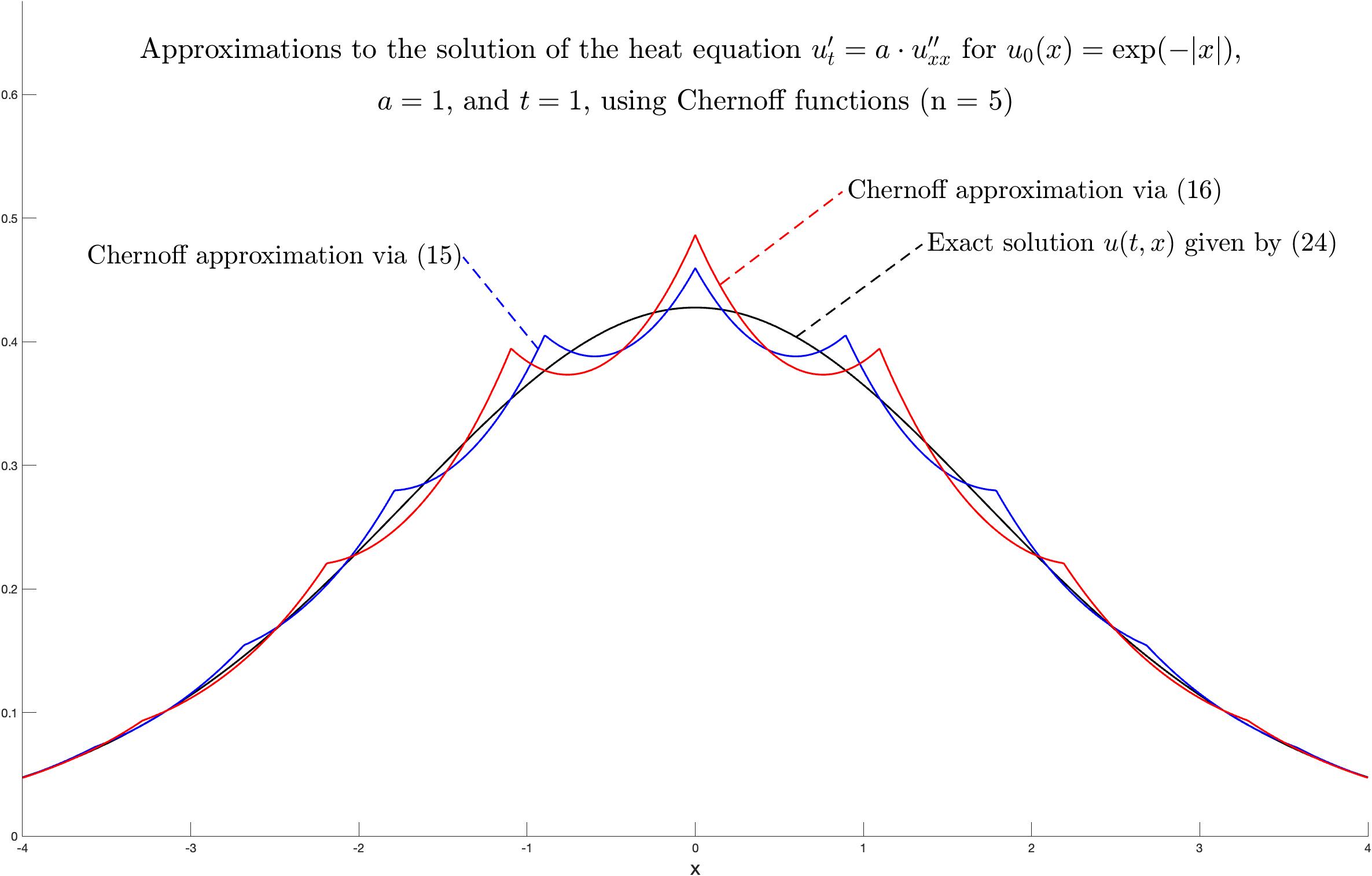}
    \caption{Graphs of approximations to the solution of the heat equation ($t = 1$, $n = 5$)}
    \label{fig:heat_exp_apprx_5}
\end{figure}

Figures \ref{fig:heat_exp_apprx_1} and \ref{fig:heat_exp_apprx_5} demonstrate two model examples of Chernoff approximations for the initial condition $u_0=[x\mapsto e^{-|x|}]$ and $t = 1$ with composition degree $n = 1$ and $n = 5$ respectively. 
\begin{figure}[H]
    \centering
    \includegraphics[width=0.8\textwidth]{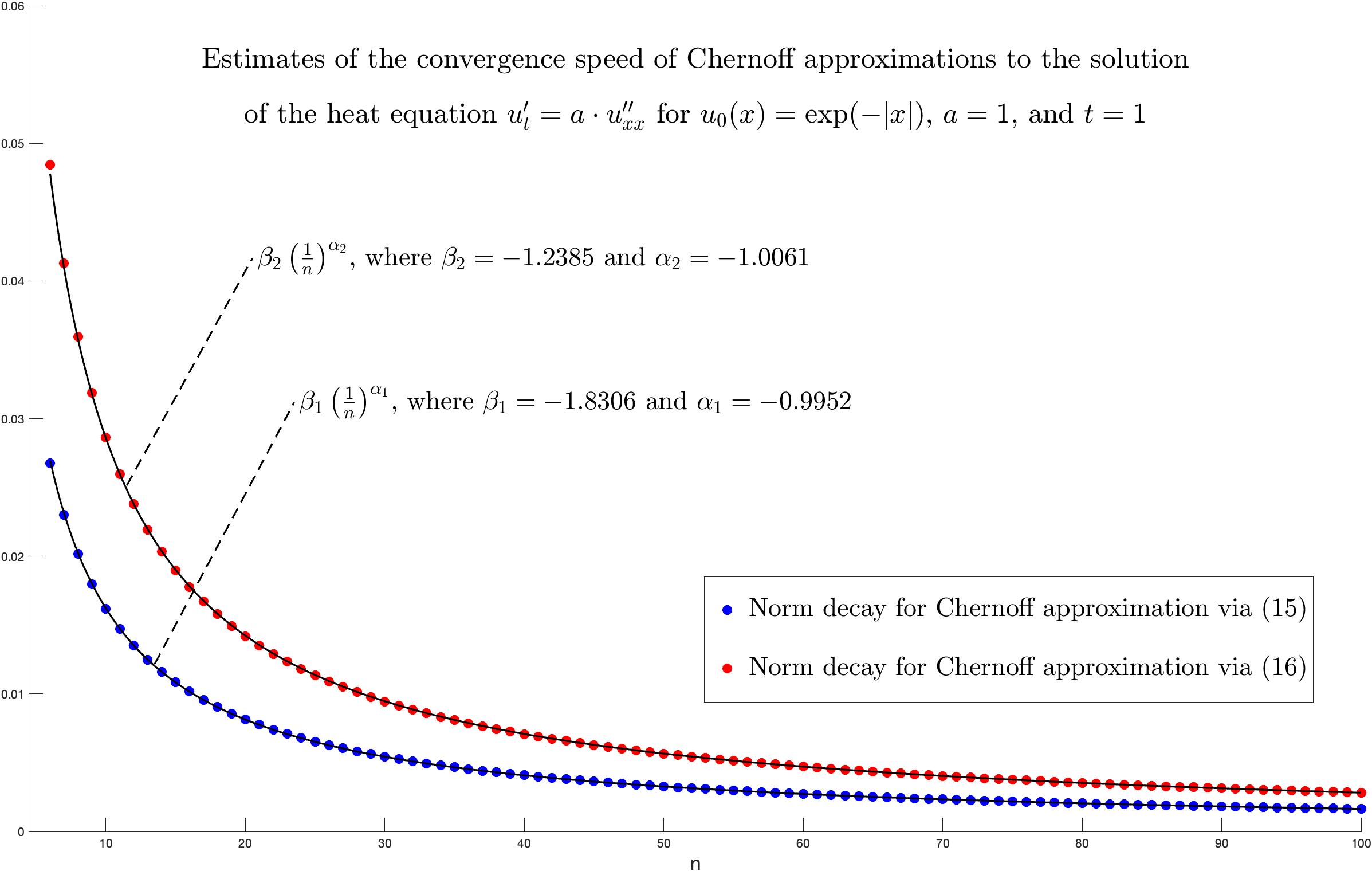}
    \caption{Graphs of estimates of the convergence speed ($t = 1$)}
    \label{fig:heat_exp_apprx_gen}
\end{figure}

Let us once again refer to the decaying property of the initial condition and the Chernoff approximations. We propose that the standard norm in $UC_b(\mathbb{R})$ is reached on the appropriately defined interval (e.g., $[-5, 5]$). So, Figure \ref{fig:heat_exp_apprx_gen} contains two graphs of convergence rate for $t = 1$.

\begin{figure}[H]
    \centering
    \includegraphics[width=0.8\textwidth]{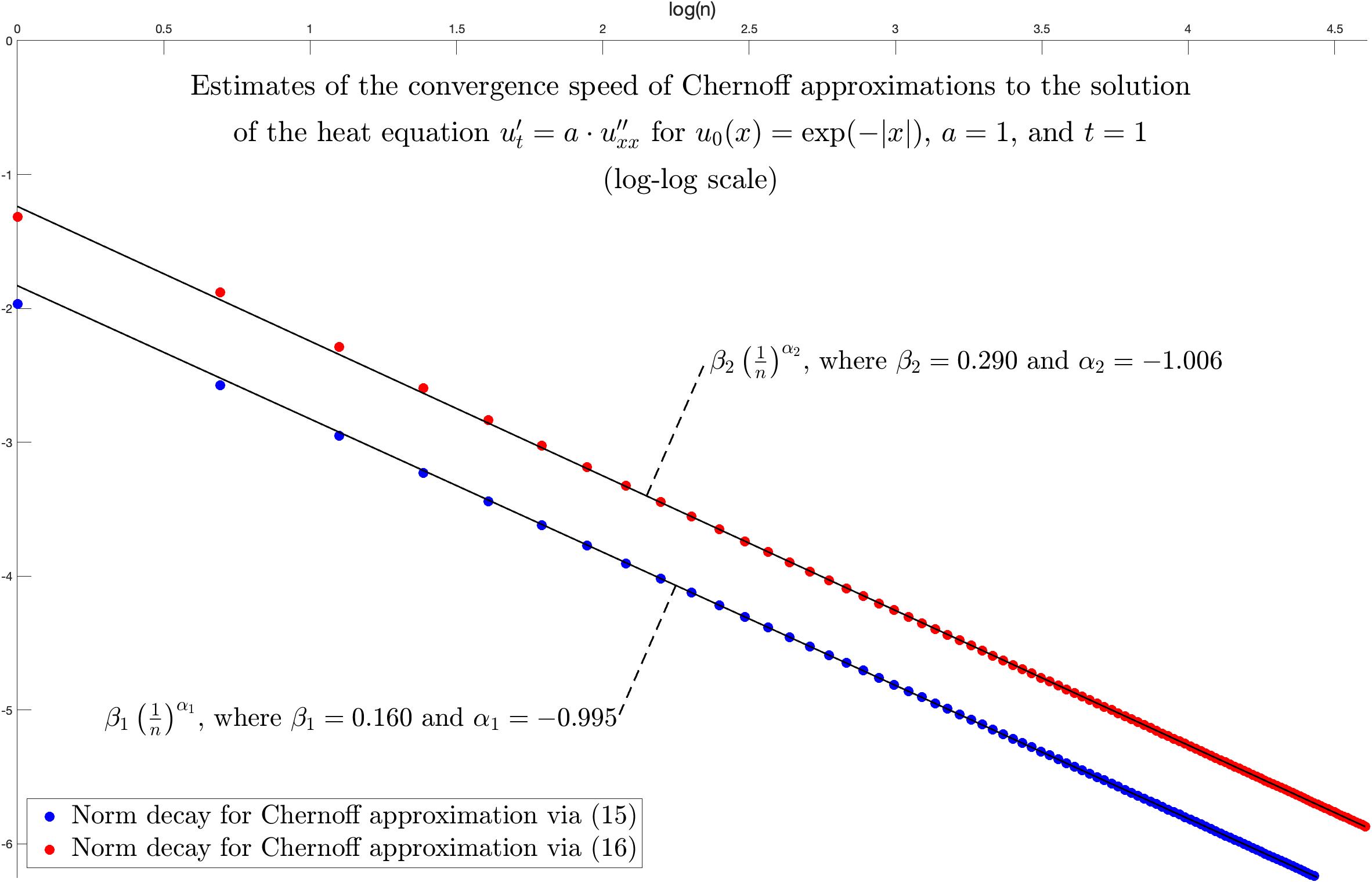}
    \caption{Graphs of estimates of the convergence speed with log-log scale ($t = 1$)}
    \label{fig:heat_exp_apprx_log}
\end{figure}

The above series of figures illustrates the results of a non-linear regression applied to analyse the approximation subspaces to which the examined initial condition belongs to. So, Figures \ref{fig:heat_exp_apprx_gen} and \ref{fig:heat_exp_apprx_log}, using both standard and logarithmic scale, represent the convergence rate and the corresponding obtained regression results for the heat equation.

One can easily derive now that the behaviour of the Chernoff approximations depends on the choice of initial condition, e.g., for $u_0 = [x\to\exp(-|x|)]$ we have it belonging to the approximation subspaces of the same order $\frac{1}{n}$ when applying expressions via different Chernoff functions \eqref{eq:chernoff_heat_first}, \eqref{eq:chernoff_heat_second}, and \eqref{eq:chernoff_heat_third}.

\bigskip
\textbf{Acknowledgements.} The author would like to express deep gratitude to his scientific supervisor I.D.~Remizov for the problem setting and attention to the research, and to O.E.~Galkin, and other members of research group 'Evolution semigroups and their applications' for useful critiques. The publication was prepared within the framework of the Academic Fund Program at HSE University in 2020–2021 (grant No.20-04-022, project Evolution semigroups and their applications) and within the framework of the Russian Academic Excellence Project '5-100'.

\printbibliography

\appendix
\section{Code for numerical experiments}
https://gitlab.com/tervenar/matlab

\end{document}